\documentclass[reqno]{amsart}

\usepackage{color}
\usepackage[dvipsnames]{xcolor}

\definecolor{old}{rgb}{0, 0.2, 0.4}

\definecolor{div1}{RGB}{0,66,157}
\definecolor{div2}{RGB}{57,100,173}
\definecolor{div3}{RGB}{85,137,189}
\definecolor{div4}{RGB}{109,175,204}
\definecolor{div5}{RGB}{130,214,219}
\definecolor{div6}{RGB}{255,202,185}
\definecolor{div7}{RGB}{253,146,145}
\definecolor{div8}{RGB}{231,93,111}
\definecolor{div9}{RGB}{197,42,82}
\definecolor{div10}{RGB}{147,0,58}
\definecolor{div11}{RGB}{0,157,58}

\definecolor{pal1}{RGB}{113,31,134} 
\definecolor{pal2}{RGB}{146,150,238}
\definecolor{pal3}{RGB}{60,63,133}
\definecolor{pal4}{RGB}{83,198,239}
\definecolor{pal5}{RGB}{17,103,126}

\definecolor{pal21}{RGB}{105,176,95} 
\definecolor{pal22}{RGB}{43,131,186}

\definecolor{matlab1}{RGB}{0,113.9850,188.9550}
\definecolor{matlab2}{RGB}{216.7500,82.8750,24.9900}
\definecolor{matlab3}{RGB}{236.8950,176.9700,31.8750}
\definecolor{matlab4}{RGB}{125.9700,46.9200,141.7800}
\definecolor{matlab5}{RGB}{118.8300,171.8700,47.9400}
\definecolor{matlab6}{RGB}{76.7550,189.9750,237.9150}
\definecolor{matlab7}{RGB}{161.9250,19.8900,46.9200}
\definecolor{matlab8}{RGB}{255.0000,68.8500,57.8850}
\definecolor{matlab9}{RGB}{100.9800,129.7950,252.9600}
\definecolor{matlab10}{RGB}{255.0000,213.9450,9.9450}
\definecolor{matlab11}{RGB}{0,162.9450,162.9450}
\definecolor{matlab12}{RGB}{202.9800,131.8350,92.8200}

\definecolor{seq11}{RGB}{253,190,133}
\definecolor{seq12}{RGB}{253,141,60}

\definecolor{seq21}{RGB}{158,154,200}
\definecolor{seq22}{RGB}{106,81,163}

\definecolor{mygray}{RGB}{113,113,113}

\usepackage{xargs}
\usepackage{expl3}

\usepackage[utf8]{inputenc}				
\usepackage[english]{babel}
\usepackage[T1]{fontenc}				

\usepackage{stix2}
\usepackage{amsmath,amsthm,amsfonts}
\usepackage{mathtools}
\usepackage{bm}
\usepackage{nicefrac}
\usepackage{extarrows}
\usepackage{leftidx}
\usepackage{accents}

\usepackage{wrapfig}
\usepackage{float}
\usepackage{subcaption}
\usepackage{epsfig}
\usepackage{epstopdf}
\usepackage{svg}
\usepackage{pdfpages}
\usepackage{graphicx}
\usepackage{tikz}
\usetikzlibrary{math,arrows,intersections}
\usepackage{transparent} 
\usepackage{pgfplots}
\pgfplotsset{compat = newest}
\usepackage{pgfplotstable}
\pgfplotstableset{col sep=comma}

\usepackage{listings}
\usepackage{algorithm}
\usepackage{algpseudocode}

\usepackage{ltxtable}
\usepackage{tabularx}
\usepackage{multirow}
\usepackage{longtable}
\usepackage{booktabs}

\usepackage{paralist}

\usepackage{url}
\usepackage[colorlinks = false,
			pdfpagelabels,
			pdfstartview = FitH,
			bookmarksopen = true,
			bookmarksnumbered = true,
			linkcolor = black,
			plainpages = false,
			hypertexnames = false,
			citecolor = black,
			hidelinks] {hyperref}

\usepackage[
backend=biber,
style=numeric,
sorting=nyt,
giveninits=true,
doi=false,
isbn=false,
url=true,
maxbibnames=7,
language=english,
]{biblatex}
\renewbibmacro{in:}{}
\usepackage{cleveref}

\usepackage{lipsum}
\usepackage[defaultcolor=orange]{changes}
\usepackage{csquotes}

\makeatletter
\newcommand\RedeclareMathOperator{%
  \@ifstar{\def\rmo@s{m}\rmo@redeclare}{\def\rmo@s{o}\rmo@redeclare}%
}
\newcommand\rmo@redeclare[2]{%
  \begingroup \escapechar\m@ne\xdef\@gtempa{{\string#1}}\endgroup
  \expandafter\@ifundefined\@gtempa
     {\@latex@error{\noexpand#1undefined}\@ehc}%
     \relax
  \expandafter\rmo@declmathop\rmo@s{#1}{#2}}
\newcommand\rmo@declmathop[3]{%
  \DeclareRobustCommand{#2}{\qopname\newmcodes@#1{#3}}%
}
\@onlypreamble\RedeclareMathOperator
\makeatother
\newcommand{\RedeclarePairedDelimiter}[3]{
	\let\angles\relax\DeclarePairedDelimiter{#1}{#2}{#3}}
	
\newcommand {\Onot}	{\mathcal{O}}   

\newcommand {\R}	{\mathbb{R}}
\newcommand {\N}	{\mathbb{N}}

\newcommand {\cP}   {\mathcal{P}}
\newcommand {\cB}   {\mathcal{B}}

\newcommand {\cF}   {\mathcal{F}}

\DeclarePairedDelimiter  {\norm} {\lVert} {\rVert}
\DeclarePairedDelimiter  {\abs}  {\lvert} {\rvert}

\DeclarePairedDelimiter  {\inner}{\langle}{\rangle}

\RedeclarePairedDelimiter{\angles}  {\langle}{\rangle}

\DeclareMathOperator* {\argmax}{arg\,max}

\DeclareMathOperator  {\spann} {span}

\RedeclareMathOperator{\closure}{clos}
\RedeclareMathOperator{\div}   {div}

\DeclareMathOperator  {\dist}{dist}

\newcommand\bomega{\boldsymbol{\omega}}

\newcommand\restr[2]{{
  \left.\kern-\nulldelimiterspace 
  #1 
  \vphantom{\big|} 
  \right|_{#2} 
  }}
\makeatletter
\renewcommand*\env@matrix[1][*\c@MaxMatrixCols c]{%
  \hskip -\arraycolsep
  \let\@ifnextchar\new@ifnextchar
  \array{#1}}
\makeatother

\makeatletter
\newsavebox\myboxA
\newsavebox\myboxB
\newlength\mylenA

\newcommand*\xoverline[2][0.75]{%
    \sbox{\myboxA}{$\m@th#2$}%
    \setbox\myboxB\null
    \ht\myboxB=\ht\myboxA%
    \dp\myboxB=\dp\myboxA%
    \wd\myboxB=#1\wd\myboxA
    \sbox\myboxB{$\m@th\overline{\copy\myboxB}$}
    \setlength\mylenA{\the\wd\myboxA}
    \addtolength\mylenA{-\the\wd\myboxB}%
    \ifdim\wd\myboxB<\wd\myboxA%
       \rlap{\hskip 0.5\mylenA\usebox\myboxB}{\usebox\myboxA}%
    \else
        \hskip -0.5\mylenA\rlap{\usebox\myboxA}{\hskip 0.5\mylenA\usebox\myboxB}%
    \fi}
\makeatother

\allowdisplaybreaks  

\definecolor{bh}{RGB}{0, 106, 255} 
\definecolor{bh2}{RGB}{0, 128, 0} 
\newcommand	{\hue}	   [1]{{\color{bh}{#1}}} 

\definechangesauthor[name=NR, color=green!70!black]{NR}

\graphicspath{{Images/}}

\newtheorem{theorem}{Theorem}[section]
\newtheorem{corollary}[theorem]{Corollary}
\newtheorem{lemma}[theorem]{Lemma}

\theoremstyle{definition}

\newtheorem{example}[theorem]{Example}

\theoremstyle{remark}
\newtheorem{remark}[theorem]{Remark}

\numberwithin{equation}{section}

\pgfplotsset{
  log x ticks with fixed point/.style={
      xticklabel={
        \pgfkeys{/pgf/fpu=true}
        \pgfmathparse{exp(\tick)}%
        \pgfmathprintnumber[fixed relative, precision=3]{\pgfmathresult}
        \pgfkeys{/pgf/fpu=false}
      }
  },
  log y ticks with fixed point/.style={
      yticklabel={
        \pgfkeys{/pgf/fpu=true}
        \pgfmathparse{exp(\tick)}%
        \pgfmathprintnumber[fixed relative, precision=3]{\pgfmathresult}
        \pgfkeys{/pgf/fpu=false}
      }
  }
}

\begin{document}

\author[Reich]{Niklas Reich}
\address{Ulm University, Institute for Numerical Mathematics, Helmholtzstr.\ 20, 89081 Ulm, Germany}
\curraddr{Hochschule Ruhr West, Institut Naturwissenschaften, Duisburger Str.\ 100, 45479 M\"ulheim a.d.\ Ruhr, Germany}
\email{niklas.reich@uni-ulm.de}

\author[Urban]{Karsten Urban}
\address{Ulm University, Institute for Numerical Mathematics, Helmholtzstr.\ 20, 89081 Ulm, Germany}
\email{karsten.urban@uni-ulm.de}

\author[Vorloeper]{Jürgen Vorloeper}
\address{Hochschule Ruhr West, Institut Naturwissenschaften, Duisburger Str.\ 100, 45479 M\"ul\-heim a.d.\ Ruhr, Germany}
\email{juergen.vorloeper@hs-ruhrwest.de}

\title[A parallel batch greedy algorithm]{A parallel batch greedy algorithm in reduced basis methods: Convergence rates and numerical results}%
\thanks{The authors have no competing interests to declare that are relevant to the content of this article. We are very grateful to Stephan Rave and Mario Ohlberger (University of Münster) for the stimulating discussions and valuable input. We thank the anonymous reviewers for very substantial input. KU gratefully acknowledges support within the project \enquote{DigiPr\"uF} funded by the German Federal Ministry for Economic Affairs and Climate Action based on a resolution of the German Bundestag as part of the \enquote{Future Investments in Vehicle Manufacturers and Supplier Industry} program.}

\subjclass[2020]{
65N15, 
65N30, 
65Y05
}%
\date{\today}

\begin{abstract}
The \enquote{classical} (weak) greedy algorithm is widely used within model order reduction in order to compute a reduced basis in the offline training phase: An a posteriori error estimator is maximized and the snapshot corresponding to the maximizer is added to the basis. Since these snapshots are determined by a sufficiently detailed discretization, the offline phase is often computationally extremely costly.

We suggest to replace the serial determination of one snapshot after the other by a parallel approach. In order to do so, we introduce a batch size $b$ and compute $b$ snapshots in parallel. These functions are added to the current reduced basis either if they satisfy a bulk criterion or after a Proper Orthogonal Decomposition (POD).

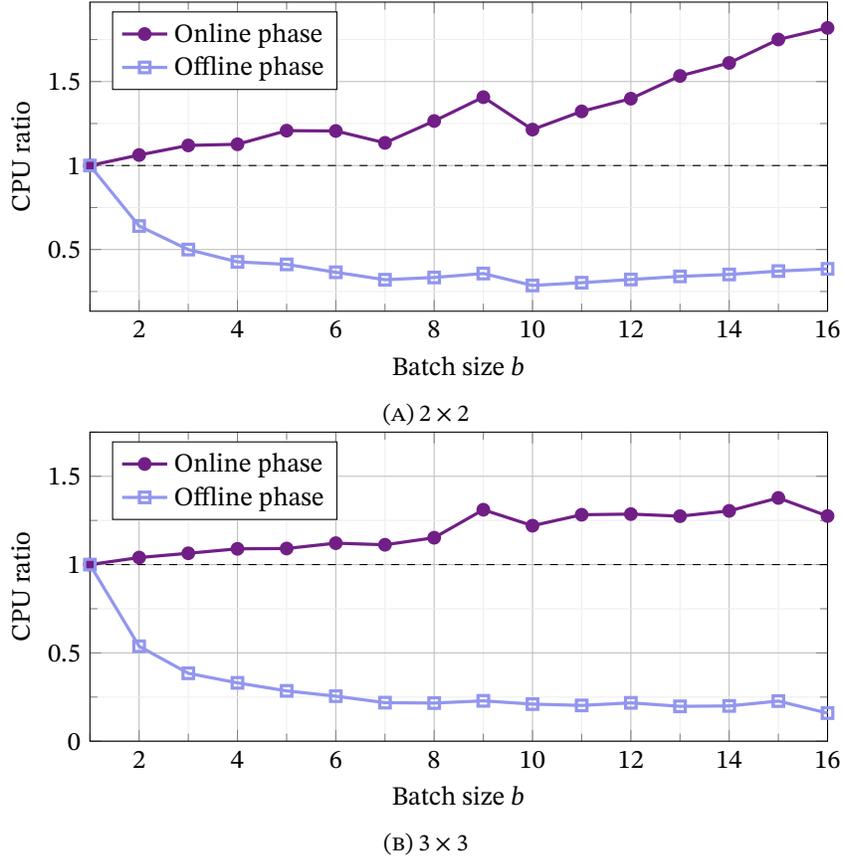
We prove convergence rates for this new batch greedy algorithm and compare them to those of the classical (weak) greedy algorithm in the Hilbert and Banach space case. Then, we present numerical results where we apply a (parallel) implementation of the proposed algorithm to the linear elliptic \emph{thermal block} problem. We analyze the convergence rate as well as the offline and online wall-clock times for different batch sizes. We show that the proposed variant can significantly speed-up the offline phase while the size of the reduced problem is only moderately increased. The benefit of the parallel batch greedy increases for more complicated problems.
\end{abstract}

\maketitle

\section{Introduction}\label{Sec:1}
Model order reduction of parameterized partial differential equations (PPDEs) by the reduced basis method (RBM) has been a very active  research field over the last at least 15 years, see \cite{Haasdonk:RB,hesthaven2016certified,quarteroni2015reduced,Urban:RB} for surveys. The reduced model is determined in an offline training phase by selecting certain samples of the parameter and using sufficiently detailed numerical solutions to compute approximations of the solution of the PPDE (called \emph{snapshots}) for those samples. These snapshots then form the reduced basis. Since the snapshots need to be sufficiently detailed, their determination might be computationally costly. Parallel computing can of course be used for each snapshot.

However, since the sample values of the parameter are determined by maximizing an a posteriori error over a training set of parameters in a (weak) greedy algorithm, one snapshot is selected in each greedy iteration, which is a serial process. In order to fully benefit from a potential gain of computational power in the offline phase, we suggest in this paper to chose $b\ge 2$ samples at once and to compute all such $b$ snapshots in parallel. We call this a \emph{parallel batch greedy algorithm}. These $b$ snapshots are either chosen by the $b$ largest values of an error estimator or as the largest and $b-1$ random samples. An alternative approach has been presented in \cite{MultiFidelity}, where the high-fidelity error estimator in the greedy algorithm is replaced by multi-fidelity error estimation.

The convergence properties of the (standard) greedy algorithm are very well un\-der\-stood, \cite{weakgreedy,weakgreedy2}. In fact, it has been proven in these papers that the convergence rate of the weak greedy method is optimal as compared to the rate of the best possible linear approximation, which is given by the Kolmogorov $n$-width. We analyze the effect of the batch variant both theoretically and numerically. It turns out that this is not completely straightforward as adding all elements of a batch may add too little information to yield a sufficient decay of the error. This is due to the fact that the functions in the batch might be (almost) linearly dependent. To overcome this issue, we consider two variants: (i) snapshots from the batch are only added to the reduced basis if they satisfy a \emph{bulk criterion} concerning the error or (ii) a Proper Orthogonal Decomposition (POD) is performed to the projection error of the batch and  only those POD modes are added who significantly contribute to reducing the error. In both cases we need to adapt  the theoretical results and the proofs in \cite{weakgreedy,weakgreedy2} to the batch case.
As expected, the convergence rate of the weak greedy (slightly) suffers for adding more snapshots from a batch, which means that the size of the reduced model increases (the offline speedup comes at the cost of additional online cost). Since this effect is based upon the sizes of involved constants, we performed numerical experiments comparing the batch greedy with the standard, \enquote{classical} greedy method. These experiments are done for a model problem as this allows us to fully compare the batch with the classical weak greedy (CWG) method. However, we clearly see that the batch version has significant potential to speed up the offline phase in particular for challenging problems. For those cases, however, required hyper-reduction would interfere with the desired comparison of classical and batch greedy method.

This paper is organized as follows. In Section \ref{Sec:Greedy}, we introduce the classical greedy method in strong and weak form. We also introduce the batch version of the weak greedy scheme. The convergence analysis is described in Section \ref{Sec:Analysis} and the results of our numerical experiments are presented in Section \ref{Sec:NumExp}.  The error analysis in Banach spaces as a straightforward generalization of \cite[\S 4]{weakgreedy2} is presented in Appendix \ref{Sec:Appendix_Banach}.

\section{Greedy algorithms}
\label{Sec:Greedy}
We start by recalling known facts on greedy algorithms from \cite{weakgreedy,weakgreedy2} and introduce the idea of a batch greedy algorithm. We restrict ourselves to the case where $X$ is a Hilbert space with a norm induced by an inner product, i.e. $\norm{\cdot} \equiv \norm{\cdot}_X \equiv\sqrt{\inner{\cdot, \cdot}}$, $\inner{\cdot, \cdot}\equiv \inner{\cdot, \cdot}_X$; the Banach space case is described in Appendix \ref{Sec:Appendix_Banach}. Let $\cF\subseteq X$ be compact and (for notational convenience only and without loss of generality) we shall assume that the elements $f$ of $\cF$ satisfy $\norm{f}_X \leq 1$, i.e., $\cF\subseteq\cB_1(0)$, the unit ball in $X$.

\subsection{The classical strong greedy algorithm}\label{Sec:csg}

Before formulating the algorithm, we introduce some notation. Let $V_n\subset X$ be some subspace of dimension $n\in\N$. Then,
\begin{align}\label{eq:sigmaf}
	\sigma_n (f):= \dist(f,V_n)  := \inf_{g_n\in V_n} \norm{f - g_n},
	\qquad
	\sigma_n \equiv\ \sigma_n (\cF):= \max_{f\in\cF}\dist(f,V_n).
\end{align}
This means in particular that $\sigma_0 (f) := \norm{f}$ and therefore $\sigma_0 (\cF):= \max_{f\in\cF} \norm{f}\le 1$, since $\dim V_0 = 0$.
The strong greedy method is shown in Algorithm \ref{alg:sg}.
\begin{algorithm}
\caption{Strong Greedy Algorithm}\label{alg:sg}
\begin{algorithmic}[1]
	\Require $V_0:=\emptyset$.
	\For{$n=0,1,2,...$} 
		\State Choose $f_n \in \cF$ such that $f_n=\argmax_{f \in \cF} \sigma_n (f)$.\label{line:eq:strong}
		\State $V_{n+1} := \spann (V_n\oplus \{ f_{n} \})$.
	\EndFor
\end{algorithmic}
\end{algorithm}

Obviously, line \ref{line:eq:strong} means that $f_n$ minimizes the distance to $V_n$, which explains the name \emph{strong} greedy.

\subsection{The Kolmogorov \emph{n}-width}\label{Sec:KolN}
We are interested in the decay of $\sigma_n$ as $n\to\infty$, i.e., the question how well $\cF$ is approximated by $V_n$ in an $L_\infty(\cF)$-sense. Since $V_n$ is a linear space, the best we can achieve is expressed by the \emph{Kolmogorov $n$-width} $d_n(\cF)$ of $\cF$ defined for $n\in\N$ by
\begin{equation*}
    d_n \equiv d_n(\cF) 
    := \inf_{\substack{X_n\subset X \\ \dim(X_n)=n}}
    \sup_{f \in \cF} \dist(f,X_n),
\end{equation*}
and we set $d_0 \equiv d_0(\cF) := \max_{f \in \cF} \norm{f} =: \sigma_0(\cF) \le 1$. If the infimum in the definition of $d_n$ is attained, the corresponding argument $X_n$ is called \emph{optimal in the sense of Kolmogorov}. We aim that the spaces $V_n$ generated by the greedy method are possibly close to such optimal space $X_n$. Hence, we would hope that $\sigma_n = \Onot(d_n)$, see \cite{weakgreedy,weakgreedy2} and references therein for corresponding results.

\subsection{The classical weak greedy algorithm}\label{Sec:cwg}
Realizing line \ref{line:eq:strong} in algorithm \ref{alg:sg}, namely determining the distance w.r.t.\ $\cF$ is either impossible or at least computationally too costly. In fact, in a reduced basis framework, the greedy method selects sample parameters $\mu^{(n)}$ and the elements are then the numerical detailed solution $f_n=u(\mu^{(n)})$ (also called \emph{snapshots}) for the chosen parameter. In the strong variant of the greedy algorithm we would need to compute the snapshot of every parameter $\mu$ in a training set. This is not feasible in practice and therefore the strong greedy method is mainly of theoretical interest and not appropriate for numerical computations. To circumvent this problem, we introduce a weak variant of the greedy algorithm, described in Algorithm \ref{alg:wg}.
\begin{algorithm}
\caption{Weak Greedy Algorithm}\label{alg:wg}
\begin{algorithmic}[1]
	\Require Fix a constant $0 < \gamma \leq 1$, $V_0:=\emptyset$.
	\For{$n=0,1,2,3,...$} 
		\State Choose $f_n \in \cF$ such that $\sigma_n(f_n) \geq \gamma\, \sigma_n(\cF)$ \label{line:wg_select}
		\State $V_{n+1} := \spann (V_n\oplus \{ f_{n} \})$.
	\EndFor
\end{algorithmic}
\end{algorithm}

In the weak form, one is able to make the choice of $f_n$ without computing it beforehand. This is usually realized by an error surrogate $r_n(f)$ satisfying 
\begin{align}\label{eq:surrogate}
	c_r\, r_n(f) \leq \sigma_n(f) \leq C_r\, r_n(f) 
	\quad\text{ for all } f \in \cF, 
\end{align}
with some fixed constants $0<c_r \le C_r <\infty$. Line \ref{line:wg_select} of Algorithm \ref{alg:wg} is then realized by setting $\gamma = \frac{c_r}{C_r}$ and choosing $f_n$ such that $f_n = \argmax_{f \in \cF} r_n(f)$, \cite{weakgreedy}. Only the chosen snapshot $f_n=u(\mu^{(n)})$ is then computed. Note that for $\gamma =1$, the weak greedy algorithm coincides with the strong greedy algorithm. Moreover, neither the greedy algorithm nor the weak greedy algorithm gives a unique sequence $(f_n)_{n \geq 0}$; also the sequence $(\sigma_n(\cF))_{n \geq 0}$ is not unique. In what follows, the notation reflects any sequence which can arise in the implementation of the weak greedy selection for the fixed value of $\gamma$.

\subsection{A weak batch greedy algorithm with bulk criterion}\label{Sec:wbg}
Our aim is to in\-cor\-por\-ate (more) parallel processing into the (weak) greedy algorithm. Of course, as mentioned before, parallel computing can be used for each single snapshot com\-pu\-ta\-tion, but the outer greedy algorithm will always remain serial. To overcome this, if in one iteration $\ell=0,1,2,...$ of a greedy algorithm $b\ge 2$ sample parameters $\mu^{(\ell,k)}$, $k=0,...,b-1$, are selected, the corresponding snapshots can be computed in parallel, which potentially significantly improves the computational efficiency. As already pointed put earlier, adding all $b$ snapshots to the reduced basis may add only too little information as those functions might be (almost) linearly dependent (see Example \ref{ex:worstcase}). To overcome this issue, after the parallel computation of the snapshots, they are only added step by step to the reduced space if a \enquote{quality} (or \enquote{bulk}) criterion is satisfied. To formulate this criterion, we introduce a \emph{bulk parameter} $0 \leq \lambda \leq 1$ and add a snapshot of the batch to the reduced space only if its error surrogate is larger than a bulk of the current error. This describes the idea of a \emph{batch algorithm} with \emph{batch size} $b$ and \emph{bulk parameter} $\lambda$, which is shown in detail in Algorithm \ref{alg:wbg}.

\begin{algorithm}
\caption{Weak Batch Greedy Algorithm with bulk criterion}\label{alg:wbg}
\begin{algorithmic}[1]
	\Require Batch size $b$, fixed constant $0 < \gamma \leq 1$, bulk parameter $0 \leq \lambda \leq 1$.
	\State Set $n=0$ and $V_n = \emptyset$.
	\For{$\ell=0,1,2,3,...$} 
		\State $\cB_\ell = \emptyset$; $b_\ell:=0$
		\For{$k=0, \ldots, b-1$} \Comment{choose batch} \label{line:begin_firstloop}
			\State \label{line:batchchoice} $\widetilde{f}_{{\ell,k}} := \argmax\limits_{f \in \cF \setminus \cB_\ell} r_{n}(f)$
			\State $\cB_\ell = \cB_\ell \cup \{ \widetilde{f}_{{\ell,k}}\}$
		\EndFor \label{line:end_firstloop}
		\State Compute the snapshots of $\cB_\ell$ in parallel. \Comment{compute batch}
		\For{$k=0, \ldots, b-1$}  \label{line:begin_secondloop}
			\State \label{line:batchcandidate} $\bar{f}_{{\ell,k}} := \argmax\limits_{f \in \cB_\ell} r_{n}(f)$
				\Comment{maximize batch w.r.t.\ current dimension $n$}
			\If{$r_n(\bar{f}_{\ell,k}) \geq \lambda \max\limits_{f \in \cF} r_n(f)$} \Comment{bulk criterion}\label{line:bulk}
				\State $n \to n+1$, $b_\ell\to b_\ell+1$ \Comment{Caution: $n$ is increased}\label{line:n}
				\State $V_{n} = \spann (V_{n-1} \oplus \{ \bar{f}_{\ell,k}\})$, possibly orthogonalize $V_{n}$
			\Else
				\State break \Comment{abandon the rest of the batch}
			\EndIf
		\EndFor \label{line:end_secondloop}
	\EndFor 
\end{algorithmic}
\end{algorithm}

Let us comment on some details of Algorithm \ref{alg:wbg}.

\begin{remark}
\begin{compactenum}[(a)]
	\item With the introduction of the batch, the iteration number $\ell$ of the greedy algorithm and the dimension $n$ of the reduced space do no longer coincide. In fact, the dimension $n$ is increased in line \ref{line:n} only if the bulk criterion is satisfied. In each iteration, $b_\ell$ snapshots are added, so that the dimension at the end of the second loop in line \ref{line:end_secondloop} reads $n(\ell):=b_0+b_1+\cdots+b_\ell$. For the special case $\lambda=0$, we add $b_\ell\equiv b$ snapshots in each iteration and thus get the dimension $n(\ell)=b\ell$ (i.e., $\ell = \lfloor n/b \rfloor$); for $\lambda>0$, we only know that $n(\ell) \geq \ell$.
	\item In the first loop (lines \ref{line:begin_firstloop}-\ref{line:end_firstloop}), the $b$ snapshots $\{ \widetilde{f}_{\ell,0}, \ldots, \widetilde{f}_{\ell,b-1}\}$ forming the batch $\cB_\ell$ are selected (but not yet computed) using the error surrogate $r_n$ w.r.t.\ the current dimension $n$ of the reduced space. Once $\cB_\ell$ is determined, all snapshots are computed in parallel, but not yet added to the basis.
	\item In the second loop (lines \ref{line:begin_secondloop}-\ref{line:end_secondloop}), the reduced space is updated. To this end, we first determine the batch element maximizing the error surrogate \emph{for the current dimension $n$}. Hence, $\widetilde{f}_{\ell,0} = \bar{f}_{\ell,0}$ for the first element. Moreover, this first element always satisfies the bulk criterion, i.e., the first element in the batch is always added to the reduced basis. Then, once an element is added to the reduced basis, $n$ is increased (line \ref{line:n}) so that the next maximization in line \ref{line:batchcandidate} is performed with a larger $n$ (and thus different $V_n$) as in the loop before. Hence, in general $\widetilde{f}_{\ell,k} \neq \bar{f}_{\ell,k}$ for $k\geq 1$.\\
	Moreover, increasing the dimension $n$ and updating the reduced space requires some numerical effort since the reduced model needs to be recomputed (Riesz representations, projectors), which will be seen in our numerical experiments below.
	\item If $f \in V_n$, then $\sigma_n(f)=r_n(f)=0$. Therefore, an element of the batch cannot be selected twice in the loop over the batch in line \ref{line:batchcandidate}.
	\item For $b=1$ and any $\lambda\in [0,1]$, the algorithm coincides with the classical greedy scheme.
	\hfill$\diamond$
\end{compactenum}
\end{remark}

We collect some properties of Algorithm \ref{alg:wbg} that will be relevant in the analysis.

\begin{lemma} \label{col:algprop}
	Assume \eqref{eq:surrogate}, $\gamma := \frac{c_r}{C_r}$ and let $\{ f_n\}_{n\in\N_0}$ be generated by Algorithm \ref{alg:wbg}. Then:
    \begin{enumerate}[(i)]
	\item \label{prop:firstofbatch} 
		For each iteration $\ell=0,1,2,...$, let $\bar{f}_{\ell,0}=\widetilde{f}_{\ell,0}$ be the first element of the batch $\cB_\ell$. This element is $f_{i(\ell)}$ for the index
		\begin{align*}
			i(\ell) := \begin{cases}
					0, & \text{for } \ell=0, \\
					b_0+\cdots+b_{\ell-1}, & \text{for } \ell\ge 1,
				\end{cases}
		\end{align*}
		and we have $\sigma_{i(\ell)}(f_{i(\ell)}) \geq \gamma\,\sigma_{i(\ell)}$ for all $\ell\in\N_0$.
    	\item \label{prop:batchquality}For very $n\in\N_0$ we have $\sigma_n(f_n) \geq \lambda\, \gamma\, \sigma_n$.
\end{enumerate}
\end{lemma}
\begin{proof} 
	In each iteration we add $b_\ell$ elements to the sequence $\{ f_n\}_{n\in\N_0}$, which explains that $i(\ell)$ is in fact the index of the first element of the batch $\cB_\ell$. Next, set for $n\in\N_0$
    \begin{equation*}
	f_n^r 	:= \argmax_{f \in \cF} r_n(f)\in\cF 
		\quad \text{and} \quad 
	f_n^\sigma := \argmax_{f \in \cF} \sigma_n(f)\in\cF.
    \end{equation*}
    As $f_{i(\ell)}$ is the first element of the batch $\cB_\ell$, we get $f_{i(\ell)} = f_{i(\ell)}^r$ by line \ref{line:batchchoice} of Algorithm \ref{alg:wbg}. Then, \eqref{eq:surrogate} yields
	\begin{align*}
	\sigma_{i(\ell)}(f_{i(\ell)}) 
		&\stackrel{\eqref{eq:surrogate}}{\geq} c_r\,  r_{i(\ell)}(f_{i(\ell)}) 
		= c_r\, r_{i(\ell)}(f_{i(\ell)}^r) 
		\geq c_r\, r_{i(\ell)}(f_{i(\ell)}^\sigma) \\
		&\stackrel{\eqref{eq:surrogate}}{\geq}  \frac{c_r}{C_r}\, \sigma_{i(\ell)}(f_{i(\ell)}^\sigma) 
		= \gamma\, \sigma_{i(\ell)}(\cF),
	\end{align*}
	where the the second inequality follows by the fact that $f^r_n$ is the maximizer of $r_n$ and the last equality is due to the definition of $f^\sigma_n$ as the maximizer of $\sigma_n$. This proves (\ref{prop:firstofbatch}).

	In order to prove (\ref{prop:batchquality}), we use the bulk criterion in line \ref{line:bulk} to deduce
	\begin{align*}
		\sigma_n(f_n) 
		&\stackrel{\eqref{eq:surrogate}}{\geq} c_r\, r_n(f_n) 
		\geq c_r\, \lambda\, \max_{f\in\cF} r_n(f) 
		\geq c_r\, \lambda\, r_n(f_n^\sigma) \\
		&\stackrel{\eqref{eq:surrogate}}{\geq} \frac{c_r}{C_r}\, \lambda\, \sigma_n(f_n^\sigma)
		= \lambda\, \gamma\, \sigma_n(f_n^\sigma) 
		= \lambda\, \gamma\, \sigma_n(\cF),
	\end{align*}
	which completes the proof.
\end{proof}

\subsection{A weak batch greedy algorithm with POD}\label{Sec: POD}
The bulk criterion realizes a guaranteed error decay when adding a new snapshot to the reduced basis. This guarantee comes at some cost: (i) adding new snapshots is a fully serial approach, since one snapshot is added at a time; (ii) adding a snapshot increases the dimension, which causes the fact that the reduced model has to be recomputed (i.e., projectors, Riesz representations, affine decomposition, etc.). Hence, an alternative is almost straightforward, namely to apply a Proper Orthogonal Decomposition (POD) on the projection errors of all elements of the batch and to add only the resulting POD modes whose singular values exceed a certain relative threshold $\lambda$ (we use the same notation as for the bulk parameter in order to compare the variants with bulk and POD later). The resulting scheme is shown in Algorithm \ref{alg:wbgPOD} below.

\begin{algorithm}
\caption{Weak Batch Greedy Algorithm with POD}\label{alg:wbgPOD}
\begin{algorithmic}[1]
	\Statex Replace lines \ref{line:begin_secondloop}-\ref{line:end_secondloop} in Algorithm \ref{alg:wbg} by
    	\setcounter{ALG@line}{8}
        \State $\cB_{\ell}^{\text{err}} := \left\lbrace \widetilde{f}_{{\ell,k}} - P_n \widetilde{f}_{{\ell,k}}\right\rbrace_{k=0,...,b-1}$ \Comment{compute projection error} \label{line:proj_err}
        \State $\left\lbrace \left\lbrace \bar{f}_{\ell,k}\right\rbrace_{k=0,...,b_{\ell}-1}, \left\lbrace \theta_{\ell,k}\right\rbrace_{k=0,...,b_{\ell}-1} \right\rbrace
            = \text{POD} \left( \cB_{\ell}^{\text{err}}; \mathtt{rtol} = \lambda \right)$ \Comment{compress batch} \label{line:pod}
		\State $V_{n+b_\ell} = \spann (V_{n} \oplus \left\lbrace  \theta_{\ell,k} \bar{f}_{\ell,k}\right\rbrace_{k=0,...,b_{\ell}-1})$, possibly orthogonalize $V_{n+b_\ell}$\label{Line:PODupdate}
        \State $n \to n+b_\ell$
\end{algorithmic}
\end{algorithm}

We collect some facts of Algorithm \ref{alg:wbgPOD} that will be relevant in the analysis.

\begin{remark}\label{rem:pod_alg}
\begin{compactenum}[(a)]
\item Line \ref{line:pod} in Algorithm \ref{alg:wbgPOD} gives the maximum number $b_{\ell} \in \{ 1, \ldots, b\}$ of singular values (and left singular vectors) such that for all $k=1, \ldots, b_{\ell} - 1$ it holds
    \begin{equation} \label{eq:lambda_sv}
        \theta_{\ell,k} \ge \lambda\, \theta_{\ell,0}.
    \end{equation}
\item The update in line \ref{Line:PODupdate} could of course be done by $\bar{f}_{\ell,k}$ instead of $\theta_{\ell,k} \bar{f}_{\ell,k}$. This is done for convenience in the analysis (see Lemma \ref{col:algpropPOD} and \S \ref{Sec:AnalysisPOD} below).
\item In some aspects the algorithm is similar to the \emph{POD-Greedy Method} \cite{haasdonk2013convergence,haasdonk2008reduced}. Some parts of the subsequent proofs are adaptions of ideas from \cite{haasdonk2013convergence}.
\hfill$\diamond$
\end{compactenum}
\end{remark}

\begin{lemma} \label{col:algpropPOD}
	Assume \eqref{eq:surrogate}, $\gamma := \frac{c_r}{C_r}$ and let $\{ f_n\}_{n\in\N_0}$ be generated by Algorithm \ref{alg:wbgPOD}. Then:
    \begin{enumerate}[(i)]
	\item \label{prop:firstofbatchPOD} 
	$\sigma_{i(\ell)}(f_{i(\ell)}) \geq \frac{\gamma}{\sqrt{b}}\,\sigma_{i(\ell)}$ for all $\ell\in\N_0$ with $i(\ell)$ being the index of the first element that is added to the current subspace given in Lemma \ref{col:algprop} (\ref{prop:firstofbatch}).
    	\item \label{prop:batchqualityPOD}For very $n\in\N_0$ we have $\sigma_n(f_n) \geq \frac{\lambda \, \gamma}{\sqrt{b}}\, \sigma_n$.
\end{enumerate}
\end{lemma}
\begin{proof}
  Let $X^b$ be the $b$-times product of $X$ with inner product $\inner{u,v}_b := \sum_{k=0}^{b-1} \inner{u_k, v_k}$ for $u:=(u_0, \ldots, u_{b-1}), v:=(v_0, \ldots, v_{b-1}) \in X^b$ and induced norm  $\norm{\cdot}_b := \sqrt{\inner{\cdot, \cdot}_b}$. 
Performing a full POD of $u \in X^b$, it is well-known that for the singular values $\{ \theta_k(u)\}_{k=0,...,b-1}$ it holds
\begin{equation}\label{eq:allPODsv}
	\sum_{k=0}^{b-1} \theta_k(u)^2 = \norm{u}_b^2.
\end{equation}
Hence, using the descending order of $\theta_{\ell,k}$, \eqref{eq:allPODsv} as well as line \ref{line:proj_err} of Algorithm \ref{alg:wbgPOD} and line \ref{line:batchchoice} (of Algorithm \ref{alg:wbg}) we get
\begin{align*}
	b \, \theta_{\ell,0}(\cB_{\ell}^{\text{err}})^2 
	&\ge \sum_{k=0}^{b-1} \theta_{\ell,k}(\cB_{\ell}^{\text{err}})^2 = \norm{\cB_{\ell}^{\text{err}}}_b^2 
	\ge \norm{\widetilde{f}_{{\ell,0}} - P_{i(\ell)} \widetilde{f}_{{\ell,0}}}^2 = \sigma_{i(\ell)} (\widetilde{f}_{{\ell,0}})^2
	\ge \gamma^2 \sigma_{i(\ell)}^2.
\end{align*}
For all $k=0, \ldots, b_{\ell}-1$, using that $\bar{f}_{\ell,k}$ are normalized and orthogonal to $V_{i(\ell)}$, we have  
\begin{align} \
        \sigma_{i(\ell)} (f_{i(\ell)+k}) 
        &= \sigma_{i(\ell)} (\theta_{\ell,k} \bar{f}_{\ell,k}) 
        = \norm{\theta_{\ell,k} (\bar{f}_{\ell,k} - P_{i(\ell)} \bar{f}_{\ell,k})}
        =\norm{\theta_{\ell,k} \bar{f}_{\ell,k}}
        = \theta_{\ell,k}\label{eq:err_theta}.
    \end{align}
  Combining this, we get $b \, \sigma_{i(\ell)} (f_{i(\ell)})^2 = b \, \theta_{\ell,0}(\cB_{\ell}^{\text{err}})^2 \ge \gamma^2 \sigma_{i(\ell)}^2$, which proves \eqref{prop:firstofbatchPOD}.

To prove \eqref{prop:batchqualityPOD}, express any $n \in \N_0$ as $n = i(\ell) + k$ for some unique $i(\ell)$ and $k \in \{0, b_{\ell}-1\}$. By using the fact that $\{ f_{i(\ell)+j}\}_{j=0}^k$ are pairwise orthogonal,  \eqref{eq:lambda_sv}, \eqref{eq:err_theta} and \eqref{prop:firstofbatchPOD}, we get
\begin{align*}
	\sigma_n (f_n) 
	&= \sigma_{i(\ell)+k} (f_{i(\ell)+k}) = \sigma_{i(\ell)} (f_{i(\ell)+k}) 
        = \theta_{\ell,k}
	\ge \lambda \theta_{\ell,0} = \lambda \sigma_{i(\ell)} (f_{i(\ell)})\\
	&\ge \lambda \frac{\gamma}{\sqrt{b}} \sigma_{i(\ell)}
         \ge \frac{\lambda \gamma}{\sqrt{b}} \sigma_{i(\ell)+k} = \frac{\lambda \gamma}{\sqrt{b}} \sigma_n,
    \end{align*}
    which concludes the proof.
\end{proof}

\section{Error analysis}
\label{Sec:Analysis}
We are now going to present the error analysis of both variants of the batch greedy algorithm. This reduces mainly to minor modifications of the proofs already published in \cite{weakgreedy,weakgreedy2}. The modifications are highlighted in \hue{blue} and essentially mark the influence of the parameter $\lambda$ and the batch size $b$. We analyze the bulk criterion version in \S \ref{Sec:AnalysisBulk} for $\lambda >0$ and the special case $\lambda=0$ in \S \ref{Sec:Analysis0}. The POD version will be analyzed in \S \ref{Sec:AnalysisPOD}. Here, we will assume that $X$ is a Hilbert space. The generalization of Algorithm \ref{alg:wbg} to general Banach spaces can be found in Appendix \ref{Sec:Appendix_Banach}. We consider the weak versions of the batch greedy algorithms, but choosing $\gamma=1$ yields corresponding results for the strong variant.

With $f_n$, $n \geq 0$ we will always denote the snapshots that span the reduced space (in contrast to the notation of $\widetilde{f}$ and $\bar{f}$ in Algorithm \ref{alg:wbg} and \ref{alg:wbgPOD}). Since in general the greedy algorithm does not terminate, we  set $f_m := 0$ for $m>N$, if the algorithm terminates at $N$, i.e., if $\sigma_N(\cF)=0$. By $(f_n^*)_{n \geq 0}$, we denote the orthogonal system obtained from $(f_n)_{n\in\N_0}$ e.g.\ by Gram-Schmidt. The orthogonal projector $P_n: X\to V_{n}$ is given by $P_n f = \sum_{i=0}^{n-1} \inner{f,f^*_i}\,f^*_i$, and, in particular,
\begin{equation*}
    f_{n} = P_{n+1}\ f_n= \sum_{j=0}^n a_{n,j}f^*_j, 
    \quad\text{where  } a_{n,j} = \inner{f_n,f^*_j}, \; j \leq n.
\end{equation*}
Without loss of generality, we may assume that $X$ is the sequence space $\ell_2(\N_0)$ and $f^*_j=e_j$, where $e_j$ is the unit sequence, i.e., $(e_j)_i=\delta_{j,i}$. Next, consider the (infinite-dimensional) lower triangular matrix (setting $a_{i,j}:=0$ for $j > i$)
\begin{equation}\label{eq:A}
    A := (a_{i,j})_{i,j\in\N_0}, 
\end{equation}
which incorporates all the information about the greedy algorithm. This matrix representation is going to be critical in the error analysis.

\subsection{POD version}
\label{Sec:AnalysisPOD}
We start by presenting the error analysis for the POD as it turns out that the estimate for the bulk criterion version can be derived as a special case. We start by modifying \cite[p.\ 1462]{weakgreedy}, \cite[p.\ 459]{weakgreedy2}.

\begin{lemma}\label{La:PsbgPOD}
	For Algorithm \ref{alg:wbgPOD} and $A$ as in \eqref{eq:A}, we have for all $n\in\N_0$ that 
\begin{enumerate}[\bfseries ({P}1)]
    \item the diagonal elements satisfy $\frac{\hue{\lambda} \gamma}{\hue{\sqrt{b}}} \sigma_n \leq \abs{a_{n,n}} \leq \sqrt{\hue{b}}\,\sigma_{n}$;
    \item for every $m \geq n$ one has $\sum_{j=n}^m a_{m,j}^2 \leq \hue{b}\,\sigma_{n}^2$.\hfill$\qed$
\end{enumerate}
\end{lemma}
\begin{proof}
As in \cite{weakgreedy,weakgreedy2} we have $a_{n,n}^2 = \norm{f_n}^2 - \norm{P_n f_n}^2 = \norm{f_n - P_n f_n}^2$ due to orthogonality. In particular, $a_{n,n} = \inner{f_n,f_n^*} = (f_n)_{n}$ since $X=\ell_2(\N_0)$. Therefore we have $\abs{a_{n,n}} = \norm{f_n - P_n f_n} = \sigma_n(f_n)$. Next, recall from Algorithm \ref{alg:wbgPOD}  that $f_m = f_{i(\ell)+k}= \theta_{{\ell,k}}\, \bar{f}_{{\ell,k}}=\sum_{j=0}^{\hue{b}-1}\omega_{k,j}\, (\widetilde{f}_{{\ell,j}} -P_{i(\ell)} \widetilde{f}_{{\ell,j}})$ for $k=0,...,b_\ell-1$ with $\widetilde{f}_{\ell,j}\in\cF$ as in line \ref{line:batchchoice} of the algorithm (recall the enumeration in Lemma \ref{col:algprop}). Here, $\bomega_k=(\omega_{k,j})_{j=0,...,\hue{b}-1}\in\R^{\hue{b}}$, $\|\bomega_k\|_{\ell_2}=1$, is determined by the SVD in the POD.\footnote{$\norm{\cdot}_{\ell_2}$ denotes the Euclidean norm of a vector in $\R^d$.}\ We are going to show that
\begin{align}\label{eq:sigma_f_POD}
\begin{split}
    \sigma_n(f_m) \leq \sqrt{\hue{b}}\, \sigma_{n}
    \end{split}
\end{align}
for all $m \ge n$. In fact, for $n \ge i(\ell)$ we have
\begin{align*}
    \sigma_n(f_m)
	&=\left\| \sum_{j=0}^{\hue{b}-1} \omega_{k,j} \left[(\widetilde{f}_{{\ell,j}} -P_{i(\ell)} \widetilde{f}_{{\ell,j}}) - P_n (\widetilde{f}_{{\ell,j}} -P_{i(\ell)} \widetilde{f}_{{\ell,j}})\right] \right\| \\
	&=\left\| \sum_{j=0}^{\hue{b}-1} \omega_{k,j} \left[\widetilde{f}_{{\ell,j}} -P_{n} \widetilde{f}_{{\ell,j}}\right] \right\|
	 \leq \sum_{j=0}^{\hue{b}-1} |\omega_{k,j}| \max_{f \in \cF} \norm{f - P_n f} \\
	&\leq \sqrt{\hue{b}}\, \|\bomega_k\|_{\ell_2}\,\sigma_{n}(\cF) = \sqrt{\hue{b}}\, \sigma_{n}
\end{align*}
and for $n \le i(\ell)$ we conclude
\begin{align*}
    \sigma_n(f_m)
	&=\left\| \sum_{j=0}^{\hue{b}-1} \omega_{k,j} \left[(\widetilde{f}_{{\ell,j}} -P_{i(\ell)} \widetilde{f}_{{\ell,j}}) - P_n (\widetilde{f}_{{\ell,j}} -P_{i(\ell)} \widetilde{f}_{{\ell,j}})\right] \right\| \\
	&=\left\| \sum_{j=0}^{\hue{b}-1} \omega_{k,j} (\widetilde{f}_{{\ell,j}} -P_{i(\ell)} \widetilde{f}_{{\ell,j}}) \right\|
	 \leq \sqrt{\hue{b}}\, \|\bomega_k\|_{\ell_2}\,\sigma_{i(\ell)}(\cF) = \sqrt{\hue{b}}\, \sigma_{i(\ell)} \le \sqrt{\hue{b}}\, \sigma_{n},
\end{align*}
which proves \eqref{eq:sigma_f_POD}.
The special case $m=n$ proves the upper inequality in \emph{\textbf{(P1)}} since $\abs{a_{n,n}} = \sigma_n(f_n)$.
The lower inequality in \emph{\textbf{(P1)}} directly follows from Lemma \ref{col:algpropPOD} \eqref{prop:batchquality} since $\frac{\hue{\lambda}\gamma}{\sqrt{\hue{b}}}  \sigma_n \leq \sigma_n (f_n) = \abs{a_{n,n}}$.
The proof of  \emph{\textbf{(P2)}} follows by the same arguments as in \cite{weakgreedy,weakgreedy2} {and again \eqref{eq:sigma_f_POD}}: We have for $m \geq n$
\begin{align*}
	\sum_{j=n}^m a_{m,j}^2 
	= \norm{f_m - P_n f_m}^2 
	=  \sigma_n (f_m)^2
	\leq \hue{b}\, \sigma_n^2,
\end{align*}
which concludes the proof.
\end{proof}

The following result from \cite{weakgreedy2} is a key to continue with the analysis.

\begin{lemma}[{\cite[Lemma 2.1]{weakgreedy2}}] \label{lemma:main}
Let $G=(g_{i,j})\in\R^{K\times K}$, $K\in\N$, be a lower triangular matrix with rows $\mathbf{g}_1,\ldots,\mathbf{g}_K$. If $W_m\subset\R^K$, $m<K$, is an $m$-dimensional subspace of with the associated orthogonal projection $P_m$, then
\begin{equation*}
	\pushQED{\qed} 
    \prod_{i=1}^{K} g_{i,i}^2 
    \leq \left[ \frac{1}{m} 
            \sum_{i=1}^K \norm{P_m \mathbf{g}_i}_{\ell_2}^2
        \right]^m  
        \left[ \frac{1}{K-m} 
            \sum_{i=1}^K \norm{\mathbf{g}_i 
                - P_m \mathbf{g}_i}_{\ell_2}^2 
        \right]^{K-m}.
        \qedhere\popQED
\end{equation*}
\end{lemma}

Next, we adapt \cite[Thm.\ 3.2]{weakgreedy2} to our algorithm.

\begin{theorem}\label{th:wbg_hilbert}
	Let $X$ be a Hilbert space, $N \in \N_0$, $K\in\N$ and $1 \leq m<K$. Then, for Algorithm \ref{alg:wbgPOD} with constants $\hue{\lambda} \in (0,1]$ and $\gamma \in (0,1]$, we have
	\begin{equation}
	\prod_{i=1}^K \sigma_{N+i}^2 
	\leq \frac{\hue{b}^{{K+}m} }{(\hue{\lambda} \gamma)^{2K}} \left[ \frac{K}{m} \right]^m 
		\left[ \frac{K}{K-m} \right]^{K-m}
		\sigma_{N+1}^{2m}\, d_m^{2K-2m}.
	\end{equation}
\end{theorem}

\begin{proof}
The proof essentially coincides with the proof of \cite[Thm.\ 3.2]{weakgreedy2}, we only need the adapted properties from Lemma \ref{La:PsbgPOD} that change the statement. We consider the $K \times K$ matrix $G=(g_{i,j})_{i,j=1,...,K}$ which is formed by the rows and columns of $A$ with indices from $\{ N+1, \ldots, N+K\}$. Each row $\mathbf{g}_i$ is the restriction of $f_{N+i}$ to the coordinates $N+1, \ldots, N+K$. 

Let $X_m$ be an $m$-dimensional optimal subspace of $X$ in the sense of Kolmogorov, i.e., $\dist(\cF,X_m)=d_m$, $m=1,\ldots, K$. Let $\widetilde{W}$ be the linear space airing from the restriction of $X_m$ to the coordinates $N+1, \ldots, N+K$. Hence, $\dim(\widetilde{W}) \leq m$. Let $W$ be an $m$-dimensional space, $W \subset \spann \{ e_{N+1}, \ldots, e_{N+K}\}$, such that $\widetilde{W} \subset W$ and $P$ and $\widetilde{P}$ are the projections in $\R^K$ onto $W$ and $\widetilde{W}$, respectively. As in \cite[Thm.\ 3.2]{weakgreedy2} for the classical weak greedy we have
\begin{equation}\label{eq:Pg}
	\norm{P \mathbf{g}_i}_{\ell_2} 
	\leq \norm{\mathbf{g}_i}_{\ell_2} 
	\leq \sqrt{\hue{b}}\,\sigma_{N+1}, \quad i=1,\ldots,K,
\end{equation}
by using property \emph{\textbf{(P2)}}, as well as for $i=1,\ldots, K$
\begin{align}\label{eq:g_Pg}
	\norm{\mathbf{g}_i - P \mathbf{g}_i}_{\ell_2} 
	&\leq \norm{\mathbf{g}_i - \widetilde{P} \mathbf{g}_i}_{\ell_2} 
	= \dist(\mathbf{g}_i, \widetilde{W})
	\leq \dist(f_{N+i},X_m) \leq d_m.
\end{align}
From the (adapted) property \emph{\textbf{(P1)}} it follows that
\begin{equation}\label{eq:prod_p2}
	\prod_{i=1}^K \abs{a_{N+i,N+i}} 
	= \prod_{i=1}^K \sigma_{N+i}(f_{N+i}) 
	\geq \left(\frac{\hue{\lambda} \gamma}{{ \sqrt{\hue{b}} } }\right)^K \prod_{i=1}^K \sigma_{N+i}.
\end{equation}

Using \eqref{eq:prod_p2} together with Lemma \ref{lemma:main}, \eqref{eq:Pg} and \eqref{eq:g_Pg} we get
\begin{align*}
	\kern-50pt
	\prod_{i=1}^K \sigma_{N+i}^2 
	&\leq \left(\frac{{ \sqrt{\hue{b}} } }{\hue{\lambda} \gamma} \right)^{2K} \prod_{i=1}^K \abs{a_{N+i,N+i}}^2
	= \left(\frac{{ \sqrt{\hue{b}} } }{\hue{\lambda} \gamma} \right)^{2K} \prod_{i=1}^K \abs{g_{i,i}}^2\\
	&\leq \left(\frac{{ \sqrt{\hue{b}} } }{\hue{\lambda} \gamma} \right)^{2K} \left[ \frac{1}{m} 
		\sum_{i=1}^K \norm{P \mathbf{g}_i}_{\ell_2}^2\right]^m 
		\left[ \frac{1}{K-m} \sum_{i=1}^K \norm{\mathbf{g}_i - P \mathbf{g}_i}_{\ell_2}^2\right]^{K-m}\\
	&\leq \left(\frac{{ \sqrt{\hue{b}} } }{\hue{\lambda} \gamma} \right)^{2K} 
		\left[ \frac{1}{m} \sum_{i=1}^K \hue{b}\, \sigma_{N+1}^2 \right]^m 
		\left[ \frac{1}{K-m} \sum_{i=1}^K d_m^2\right]^{K-m}
		\\
	&
	= \frac{ \hue{b}^{{K+}m}  }{(\hue{\lambda} \gamma)^{2K}} \left[ \frac{K}{m} \right]^m 
		\left[ \frac{K}{K-m} \right]^{K-m}
		\sigma_{N+1}^{2m}\, d_m^{2K-2m},
\end{align*}
which completes the proof.
\end{proof}

We note some special cases as a generalization of \cite[Cor.\ 3.3]{weakgreedy2}, which generalize \cite[Thm.\ 3.1]{weakgreedy} and \cite[Thm.\ 3.2]{weakgreedy}.

\begin{theorem}\label{col:wbg_hilbert}
	For Algorithm \ref{alg:wbgPOD} with constant $\gamma \in (0,1]$, bulk parameter $\hue\lambda \in (0,1]$ and batch size $\hue b \ge 1$ we have:
	\begin{enumerate}[(i)]
	\item For any compact set $\cF$ and $n \geq 1$, we have
		\begin{equation}
		\label{eq:wbg_sigman_dn}
			\sigma_{n}(\cF) \leq \sqrt{2}(\hue{\lambda}\gamma)^{-1} 
			\min_{1 \leq m < n} \hue{b}^{\frac{{n+}m}{2n}}\, d_m(\cF)^{\frac{n-m}{m}}.
		\end{equation}
		In particular, $\sigma_{2n}(\cF) \leq \sqrt{2}\,\hue{b}^{{3}/4}\,(\hue{\lambda}\gamma)^{-1} \sqrt{d_n(\cF)}$ for $n\in\N$.
	\item If $d_n(\cF) \leq C_0 n^{-\alpha}$, $n\in\N$, then we have $\sigma_n(\cF) \leq C_1(\hue{\lambda},\hue{b})\, n^{-\alpha}$ with\\
		$
			C_1(\hue{\lambda},\hue{b}) 
				:= \max \left \lbrace \frac{2^{5\alpha + 1}}{\hue{\lambda}^2 \gamma^2} \hue{b}^{{{3}/{2}}}\,C_0,
                3^{\alpha} \right \rbrace.
		$
	\item If $d_n(\cF) \leq C_0 e^{-c_0 n^{\alpha}}$, $n\in\N$, then $\sigma_n(\cF) \leq C_1(\hue{\lambda},\hue{b})\, e^{-c_1(\hue{\lambda},\hue{b})\, n^{\alpha}}$ with
		\begin{equation*}
			 C_1(\hue{\lambda},\hue{b})
             := \max \left \lbrace \frac{\sqrt{2 C_0}\,}{\hue{\lambda} \gamma}\, \hue{b}^{{{3}}/{4}},
             1 + \epsilon \right \rbrace
             \quad \text{and} \quad
             c_1(\hue{\lambda},\hue{b}) := \min \left \lbrace 2^{-1-2\alpha}c_0,
             \ln(C_1(\hue{\lambda},\hue{b})) \right \rbrace
		\end{equation*}
        for any $\epsilon > 0$.
\end{enumerate}
\end{theorem}
\begin{proof} 
\textbf{(i)} We take, as in \cite{weakgreedy2}, $N=0$, $K=n$, and any $1 \leq m <n$ in Theorem \ref{th:wbg_hilbert}, and use the monotonicity of $(\sigma_n)_{n\geq 0}$ and the fact that $\sigma_0 \leq 1$ to obtain
		\begin{align*}
			\sigma_{n}^{2n} 
			&\leq \prod_{i=1}^{n} \sigma_{i}^2 \leq \frac{\hue{b}^{{n+}m}}{(\hue{\lambda} \gamma)^{2n}} 
			\left[ \frac{n}{m}\right]^m 
			\left[ \frac{n}{n-m}\right]^{n-m} 
			\sigma_{1}^{2m}\, d_m^{2n -2m}\\
        \intertext{which is equivalent to}
			\sigma_{n} 
			&\leq \frac{\hue{b}^{\frac{{n+}m}{2n}}}{\hue{\lambda} \gamma} 
			\left( \left[ \frac{n}{m}\right]^{\frac{m}{n}} 
			\left[ \frac{n}{n-m}\right]^{\frac{n-m}{n}} \right)^{1/2} 
			\!\!\sigma_{1}^{\frac{m}{n}} d_m^{\frac{n-m}{n}}
			\leq \frac{\hue{b}^{\frac{{n+}m}{2n}}}{\hue{\lambda} \gamma}
            \sqrt{2}\, \sigma_{0}^{\frac{m}{n}}\,d_m^{\frac{n-m}{n}}\\
			&\leq \sqrt{2}\, (\hue{\lambda} \gamma)^{-1} 
				\min_{1 \leq m < n} \hue{b}^{\frac{{n+}m}{2n}}\, d_m^{\frac{n-m}{n}},
		\end{align*}
	where we used that $x^{-x}(1-x)^{x-1} \leq 2$ for $0 <x = \frac{m}{n} < 1$.

\noindent
\textbf{(ii)} Now, we use Theorem \ref{th:wbg_hilbert} with $N=K=n$, and any $1 \leq m <n$ to obtain
	\begin{align*}
		\sigma_{2n}^{2n} 
		&\leq \prod_{j=n+1}^{2n} \sigma_{j}^2 
		\leq \frac{\hue{b}^{{n+}m}}{(\hue{\lambda} \gamma)^{2n}}
			\left[ \frac{n}{m}\right]^m 
			\left[ \frac{n}{n-m}\right]^{n-m} 
			\sigma_{n+1}^{2m}\, d_m^{2n -2m},\\
		\intertext{which is equivalent to}
		\sigma_{2n} 
		&\leq \frac{\hue{b}^{\frac{{n+}m}{2n}}}{\hue{\lambda} \gamma}
			\left( \left[ \frac{n}{m}\right]^{\frac{m}{n}} 
				\left[ \frac{n}{n-m}\right]^{\frac{n-m}{n}} \right)^{\frac{1}{2}}
				\sigma_{n}^{\frac{m}{n}} d_m^{\frac{n-m}{n}}
		\leq \frac{\sqrt{2}\,\hue{b}^{\frac{{n+}m}{2n}}}{\hue{\lambda} \gamma}\, 
			\sigma_{n}^{\frac{m}{n}}\, d_m^{\frac{n-m}{n}}.
	\end{align*}
	In the case $n=2s$ and $m=s$, we have $\sigma_{4s} \leq \sqrt{2}\, (\hue\lambda\gamma)^{-1} \hue{b}^{{3}/4}\,\sqrt{\sigma_{2s}\,d_s}$. We now prove the claim by contradiction. Suppose that the statement in (ii) does not hold and let $M$ be the first value where $\sigma_M>C_1(\hue\lambda,\hue{b})\,M^{-\alpha}$.
	
\noindent
\emph{Case 1:} $M =4 s$ with $s\geq 1$: Since $M=4s$ is the first value for which $\sigma_M>C_1(\hue\lambda,\hue{b})\,M^{-\alpha}$ holds and $2s<4s$, we have $\sigma_{2s} \leq C_1(\hue\lambda,\hue{b})\, (2s)^{-\alpha}$ and using $d_s \leq C_0 s^{-\alpha}$ yields
	\begin{equation}\label{eq:sigma_4bs}
		C_1(\hue\lambda,\hue{b})\, \left( 4 s \right)^{-\alpha} 
		< \sigma_{4s}
		\leq (\hue{\lambda} \gamma)^{-1}\, \hue{b}^{{3}/4}\, \sqrt{2\sigma_{2s}d_s}
		\leq (\hue{\lambda} \gamma)^{-1} s^{-\alpha}\, \hue{b}^{{3}/4}\, \sqrt{2^{1-\alpha} C_0 C_1(\hue\lambda,\hue{b})\,} 
	\end{equation}
	which yields the desired contradiction by the estimate $C_1(\hue\lambda,\hue{b}) < 2^{3\alpha + 1}\, \hue{b}^{{3}/2}\, (\hue{\lambda} \gamma)^{-2}  C_0 < 2^{5\alpha + 1} (\hue{\lambda} \gamma)^{-2}  \,\hue{b}^{{3}/2}\, C_0$.

\noindent
\emph{Case 2:} $M = 4 s + q$ with $q \in \left\lbrace  1, 2, 3 \right\rbrace$ for $s\geq 1$:
	Note that $4s+q < 2 \cdot 4s$. It then follows from \eqref{eq:sigma_4bs} that
	\begin{align*}
		C_1(\hue\lambda,\hue{b})\, 2^{-\alpha} \left( 4s \right)^{-\alpha} 
		&< C_1(\hue\lambda,\hue{b})\, \left( 4 s + q\right)^{- \alpha}
		< \sigma_{4 s + q} 
		\leq \sigma_{4 s}\\
		&\leq (\hue{\lambda} \gamma)^{-1} s^{-\alpha}\, 
		\,\hue{b}^{{3}/4}\,
		\sqrt{2^{1-\alpha} C_0 C_1(\hue\lambda,\hue{b})\,}.
	\end{align*}
	This yields the desired contradiction $C_1(\hue\lambda,\hue{b})\, < 2^{5\alpha + 1} (\hue{\lambda} \gamma)^{-2}  \,\hue{b}^{{3}/2}\,C_0$.

\noindent
\emph{Case 3:}  $1 \leq M \leq 3$: From the assumption $\norm{f}_X \leq 1$, the definition of $C_1(\hue\lambda,\hue{b})$ and the monotonicity of $(\sigma_n)_{n\geq 0}$ we obtain the contradiction as $1 \geq \sigma_0 \geq \sigma_M > C_1(\hue\lambda,\hue{b})\, M^{-\alpha}  \geq C_1(\hue\lambda,\hue{b})\, 3^{-\alpha}  \geq 1$.
\smallskip

\noindent
\textbf{(iii)} From (i), we have
\begin{align}\label{eq:sigma_2nb_exp}
		\sigma_{2s} 
		\leq \sqrt{2} (\hue{\lambda} \gamma)^{-1} \,\hue{b}^{{3}/4}\,\sqrt{d_s}
		\leq \sqrt{2} (\hue{\lambda} \gamma)^{-1} \,\hue{b}^{{3}/4}\,\sqrt{C_0\, e^{-c_0 s^\alpha}}
        = C_1(\hue\lambda,\hue{b})\, e^{-c_0\, 2^{-1-\alpha} (2s)^\alpha}.
\end{align}
We prove the claim again by contradiction. Suppose that the statement (iii) does not hold and let $M$ be the first value where $\sigma_M(\cF)>C_1(\hue\lambda,\hue{b})\, e^{-c_1 M^\alpha}$.\\
\emph{Case 1:} $M = 2 s$ with $s \geq 1$:  With \eqref{eq:sigma_2nb_exp}, we get the estimate
	$C_1(\hue\lambda,\hue{b})\, e^{-c_1(\hue\lambda,\hue{b})\, \left( 2 s \right)^\alpha} < \sigma_{2s} \leq C_1(\hue\lambda,\hue{b})\, e^{-c_0\, 2^{-1-\alpha} (2s)^\alpha}$. This yields the desired contradiction by comparison of the exponents, namely, $c_1(\hue\lambda,\hue{b}) > 2^{-1-\alpha} c_0 > 2^{-1-2\alpha} c_0$.\\
\emph{Case 2:} $M = 2 s + 1$ for $s \geq 1$: Note that $2 \cdot 2s > 2s+1$. It then follows from \eqref{eq:sigma_2nb_exp} that
	\begin{align*}
	C_1(\hue\lambda,\hue{b})\, e^{-c_1(\hue\lambda,\hue{b})\, \left(2 \cdot 2 s \right)^\alpha} 
	< C_1(\hue\lambda,\hue{b})\, e^{-c_1(\hue\lambda,\hue{b})\, \left( 2 s + 1\right)^\alpha} 
	&< \sigma_{2 s + 1}
	\leq \sigma_{2s}
	\leq C_1(\hue\lambda,\hue{b})\, e^{-c_0\, 2^{-1-\alpha} (2s)^\alpha}.
	\end{align*}
	Comparing the exponents yields a contradiction, i.e., $c_1(\hue\lambda,\hue{b})\, >  2^{-1-2\alpha} c_0$.\\
\emph{Case 3:} $M=1$: From the assumption that $\norm{f}_X \leq 1$ and the definition of $c_1(\hue\lambda, \hue b)\,$ we see the desired contradiction by 
	$1 \geq \sigma_0 \geq \sigma_M = \sigma_1 > C_1(\hue\lambda,\hue{b})\, e^{-c_1(\hue\lambda,\hue{b})\, 1^\alpha} \geq 1$. By definition, we have $C_1(\hue\lambda,\hue{b}) > 1$ which ensures that $c_1(\hue\lambda,\hue{b}) > 0$.
\end{proof}

\subsection{Bulk criterion version} \label{Sec:AnalysisBulk}

Next, we, present the error analysis for the bulk cri\-te\-ri\-on. It is easily seen that the estimates in Lemma \ref{La:PsbgPOD} can be sharpened.
\begin{lemma}\label{La:Psbg}
	For Algorithm \ref{alg:wbg} and $A$ as in \eqref{eq:A}, we have for all $n\in\N_0$ that 
\begin{enumerate}[\bfseries ({P}1)]
    \item the diagonal elements satisfy $\hue{\lambda} \gamma \sigma_n \leq \abs{a_{n,n}} \leq \sigma_{n}$;
    \item for every $m \geq n$ one has $\sum_{j=n}^m a_{m,j}^2 \leq \sigma_{n}^2$.
\end{enumerate}
\end{lemma}
\begin{proof}
As {$f_m\in\cF$ for every $m\in\N_0$}, we can modify \eqref{eq:sigma_f_POD} as follows:\\
 $\sigma_n(f_m)
    = \norm{f_m - P_n f_m}
    \leq \max_{f \in \cF} \norm{f - P_n f} 
    = \sigma_{n}(\cF) = \sigma_{n}$
for any $m\ge n$. The rest of the proof is analogous to Lemma \ref{La:PsbgPOD}. The lower inequality of \textbf{(P1)} follows using Lemma \ref{col:algprop} instead of Lemma \ref{col:algpropPOD}.
\end{proof}

Apparently, Lemma \ref{La:Psbg} corresponds to Lemma \ref{La:PsbgPOD} for ${b}=1$. The same holds true for Theorem \ref{th:wbg_hilbert} applied to Algorithm \ref{alg:wbg}. Consequently, the following statement is a direct consequence of Theorem \ref{col:wbg_hilbert} using ${b}=1$.

\begin{corollary}
	For Algorithm \ref{alg:wbg} with constant $\gamma \in (0,1]$ and relative tolerance $\hue \lambda \in (0,1]$ we have:
	\begin{enumerate}[(i)]
	\item If $d_n(\cF) \leq C_0 n^{-\alpha}$, $n\in\N$, then $\sigma_n(\cF) \leq C_1(\hue{\lambda})\, n^{-\alpha}$, where the constant reads
		$C_1(\hue{\lambda}) := \max \left \lbrace \frac{2^{5\alpha + 1}}{\hue{\lambda}^2 \gamma^2} C_0, 3^{\alpha} \right \rbrace$.
	\item If $d_n(\cF) \leq C_0 e^{-c_0 n^{\alpha}}$, $n\in\N$, then $\sigma_n(\cF) \leq C_1(\hue{\lambda})\, e^{-c_1(\hue{\lambda})\, n^{\alpha}}$ with the constants
		$C_1(\hue{\lambda}) := \max \left \lbrace \frac{\sqrt{2 \, C_0}}{\hue{\lambda} \gamma}, 1 + \epsilon \right \rbrace$
             as well as $c_1(\hue{\lambda}) := \min \left \lbrace 2^{-1-2\alpha}c_0,
             \ln(C_1(\hue{\lambda})) \right \rbrace$
        for any $\epsilon > 0$.
        \hfill$\qed$
\end{enumerate}
\end{corollary}

\begin{remark}\label{rem:theory}
\begin{compactenum}[(a)]
\item Instead of obtaining the error analysis of the bulk version as a special case of the POD variant, one can also adapt \cite[\S 3]{weakgreedy2} by substituting $\gamma$ by $\lambda \gamma$ to obtain the same results. Therefore, for $\lambda=1$, we recover the known result for the classical weak greedy (CWG) algorithm. 
\item The choice of the batch size ${b}$ does not directly influence the convergence rate as the bulk criterion is relevant for adding snapshots to the reduced basis. However, the choice of $\lambda$ influences how many (pre-computed) snapshots are added to the reduced space and therefore impact the runtime of the algorithm (see also \S \ref{Sec:NumExp} below).\hfill$\diamond$
\end{compactenum}
\end{remark}

In summary, both versions of the batch greedy method have the same \emph{asymptotic} rate of convergence.

\subsection{Special case: $\lambda=0$, i.e., no bulk criterion}
\label{Sec:Analysis0}
When setting $\lambda=0$, all snapshots in the batch are added to the reduced basis, i.e., there is no bulk criterion and no POD.\footnote{In fact, for $\lambda=0$ both variants produce the same space.} However, adding all elements of the batch might not be a good idea as the following example shows, which was provided to us by an anonymous referee.

\begin{example} \label{ex:worstcase}
	Let ${b}=2$ and consider
	\begin{equation*}
		\cF := \left\lbrace \frac{1}{n+1} e_n: n\in\N_0\right\rbrace \cup \left\lbrace \frac{1-\varepsilon}{n+1} e_n : n\in\N_0\right\rbrace \subset \ell_2(\N_0)
\end{equation*}
with $\varepsilon>0$ small. For this set and an appropriately chosen $\gamma$, Algorithm \ref{alg:wbg} with $\lambda=0$ will choose $f_{2n}=\frac{1}{n+1}e_n$ and $f_{2n+1} = \frac{1-\varepsilon}{n+1} e_n$, so the snapshots within the batch will be linear dependent.\hfill$\diamond$
\end{example}

\subsection{How to choose the elements in the batch?}\label{Sec:bminusone}
The above analysis relies on the fact that the \emph{first} parameter in the batch corresponds to the maximal value of the error indicator. The remaining $b-1$ elements could be chosen arbitrarily. The bulk criterion or the POD in fact ensure that only significant information is added to the reduced basis. The above convergence statements remain valid. We will consider this random variant in our subsequent numerical experiments.

Of course, one could also determine the whole batch by randomly chosen samples. In that case, however, convergence is no longer guaranteed by our above analysis.

\section{Numerical Experiments}
\label{Sec:NumExp}
Finally, we present results from some of our numerical experiments concerning the (parallel) weak batch greedy algorithm.

\subsection{Implementation and hardware}\label{sec:implhard}
Our implementation is based upon the well-known \emph{Python}-based model order reduction library named \emph{pyMOR}, \cite{pymor2, pymor}. The weak batch greedy algorithm, the error estimator and the parallel computation of the snapshots use an \emph{MPI}-based worker pool enabled by \emph{pyMOR}, \cite{mpi4py,mpi40}. Moreover, \emph{pyMOR} uses \emph{NumPy} and \emph{SciPy} with backends from \emph{OpenBLAS}, \cite{numpy,scipy,openblas}. All computations have been performed on a \emph{Linux} server with two \emph{Intel Xeon Gold 6230} CPUs with 20 cores each and 754 GB RAM. If not specified otherwise, the \emph{MPI}-based worker pool was used with 30 workers, i.e., we use up to 30 CPU cores in parallel.

The snapshot computation boils down to a sparse LU factorization using \emph{SciPy}s routine \texttt{splu}, which is based upon  \emph{SuperLU}, \cite{superlu, splu}. In turn, \emph{SuperLU} is a parallelized implementation that uses \emph{BLAS} and \emph{LAPACK} backends. Unfortunately, we observed that \texttt{splu} was fastest on a single core and got slightly slower the more cores it was allowed to use. Therefore we configurated \emph{NumPy} to run single-core and only parallelized via \emph{MPI}.

The open source software for all our experiments is available on the \texttt{github} site \url{https://github.com/niklasreich/parallel_batch_greedy}.

\subsection{Model problem}

We choose the thermal block problem, which is a well-known and widely studied model problem for model reduction, see for example \cite[Rem.\ 3.6]{quarteroni2015reduced},  \cite[\S 6.1.4]{hesthaven2016certified}. This choice allows us to investigate the effect of taking a batch in the offline phase in comparison to the standard greedy scheme.  Of course, the batch greedy algorithm is applicable (and most likely more meaningful) for more demanding problems. However, then one often needs additional components like the \emph{(discrete) empirical interpolation method}, whose consideration would pollute the differences between the batch greedy and the classical greedy algorithm. Therefore, we stick to this rather simple \enquote{fruit fly} problem.

Let $\Omega := (0,1)^2 = \cup_{p=1}^{P} \Omega_{p}$ be subdivided in $P=p_x\times p_y$ sub-blocks. We report results for $p_x=p_y=2$, $P=4$ and $p_x=2, p_y=3$, $P=6$.
The bilinear form of the corresponding Dirichlet problem reads $a_\mu(u,v):=\sum_{p=1}^P \mu_p (\nabla u, \nabla v)_{L_2(\Omega_p)}$ and choose $\mu\in \cP := [0.1, 1]^P \subset \R^P$. With respect to \S \ref{Sec:Analysis}, we have  $X = H_0^1(\Omega)$ and $\cF = \left\lbrace u(\mu) : \mu \in \cP \right\rbrace$. 

We use the \emph{pyMOR} demo \texttt{thermalblock} as a basis implementation and discretize  $\Omega$ by \emph{pyMOR} using triangular elements with a maximal diameter of $10^{-3}$. This results in $2.002.001$ degrees of freedom. This fine grid is chosen in order to ensure that the snapshot computation is significant for parallel computing. The parameter space $\cP := [0.1, 1]^P$ is discretized for $P=4$ by 25 and for $P=6$ by $10$ equidistant points per dimension, resulting in a training set of size $25^{4}$ and $10^6$, respectively. 
Corresponding to the 30 workers in our \emph{MPI}-pool (see \S\ref{sec:implhard}) we use a batch size of $b=30$ if not specified otherwise.

\subsection{Batch greedy error decay}
We start by investigating the error of the reduced problem as the number $n$ of basis functions (i.e., the dimension of the reduced problem) increases. It is known from \cite{buffa-maday-patera-prudhomme-turinici-2012,OR16} that the greedy algorithm for the thermal block converges exponentially fast w.r.t.\ $n$, i.e., $\sigma_n(\cF)\le C e^{-cn^\alpha}$ for (unknown) constants $0< \alpha, c,C$. Hence, the classical weak greedy (CWG) sets the benchmark for the batch greedy method.

Figure \ref{fig:err_decay} contains the results using the bulk criterion, where the indicated relative error $\epsilon_{\max,\text{rel}, |\cdot|_{1;\Omega}}(n)$ is maximized over a test set $\cP_{\text{test}}\subset\cP$ w.r.t.\ $| v|_{1;\Omega}:=\|\nabla v\|_{0;\Omega}$ and for $\lambda \in \{ 10^{-5}, 10^{-3}, 0.03, 1\}$ (using batch size $b=30$). Each iteration was stopped when reaching a relative tolerance of $\epsilon_{tol}=10^{-5}$ (indicated by the dashed line). The corresponding results for the POD variant are shown in Figure \ref{fig:err_decay_pod}.

\begin{figure}[!htb]
    \pgfplotstableread{Data/2x2/thermalblock_lambda_cwg.data}{\datarev}
    \pgfplotstableread{Data/2x2/thermalblock_lambda1.0.data}{\datazero}
    \pgfplotstableread{Data/2x2/thermalblock_lambda0.1.data}{\datamone}
    \pgfplotstableread{Data/2x2/thermalblock_lambda0.03.data}{\datamtwo}
    \pgfplotstableread{Data/2x2/thermalblock_lambda0.001.data}{\datamthree}
    \pgfplotstableread{Data/2x2/thermalblock_lambda0.0001.data}{\datamfour}
    \pgfplotstableread{Data/2x2/thermalblock_lambda1e-05.data}{\datamfive}
    \def\vCone{2.0260e5}
    \def\vcone{4.0823}
    \def\valpha{5.74208e-1}
    \def\voffsetdown{1e-2}
    \def\voffsetup{5e3}
\begin{subfigure}[c]{\textwidth}
    \centering
    \begin{tikzpicture}
        \begin{axis}[
            ymode=log,
            xmin = 0, xmax = 26,
            ymin = 8e-7,
            ymax = 1e2,
            ytick = {1e0, 1e-3, 1e-6, 1e-9, 1e-12},
            grid = both,
            minor tick num = 1,
            major grid style = {lightgray},
            minor grid style = {lightgray!25},
            width = 0.9\textwidth,
            height = 0.45\textwidth,
            legend cell align = {left},
            legend style={at={(0.5, 0.95)},anchor= north},
            legend columns = 5,
            xlabel = {size $n$ of the reduced basis},
            ylabel = {$\epsilon_{\max,\text{rel}, |\cdot|_{1;\Omega}}(n)$}
        ]
            \addplot[mark=none, gray, samples=2,domain=0:50, style= dashed, forget plot,thick]
            {1e-5}
            node[pos=0.05, above]{$\epsilon_{tol}$};

            \addplot[matlab4, style= thick, mark = o] table [x = {n}, y = {err}] {\datarev};
            \addlegendentry{CWG}
            
            \addplot[matlab9, style= thick, mark = star] table [x = {n}, y = {err}] {\datazero};
            \addlegendentry{$\lambda=1$}
            
            \addplot[matlab2, style= thick, mark = square] table [x = {n}, y = {err}] {\datamtwo};
            \addlegendentry{$\lambda=0.03$}
            
            \addplot[matlab3, style= thick, mark = diamond] table [x = {n}, y = {err}] {\datamthree};
            \addlegendentry{$\lambda=10^{-3}$}
            
            \addplot[matlab5, style= thick, mark = pentagon] table [x = {n}, y = {err}] {\datamfive};
            \addlegendentry{$\lambda=10^{-5}$}
        \end{axis}
    \end{tikzpicture}
    \subcaption{$2\times 2$ thermal block, bulk criterion}
    \label{fig:err_decay22}
\end{subfigure}
    \pgfplotstableread{Data/2x3/thermalblock_lambda_cwg.data}{\datarev}
    \pgfplotstableread{Data/2x3/thermalblock_lambda1.0.data}{\datazero}
    \pgfplotstableread{Data/2x3/thermalblock_lambda0.1.data}{\datamone}
    \pgfplotstableread{Data/2x3/thermalblock_lambda0.03.data}{\datamtwo}
    \pgfplotstableread{Data/2x3/thermalblock_lambda0.001.data}{\datamthree}
    \pgfplotstableread{Data/2x3/thermalblock_lambda0.0001.data}{\datamfour}
    \pgfplotstableread{Data/2x3/thermalblock_lambda1e-05.data}{\datamfive}
    \def\vCone{0.1089}
    \def\vcone{0.0499}
    \def\valpha{0.9906}
    \def\voffsetdown{3e-1}
    \def\voffsetup{6e1}
\begin{subfigure}[c]{\textwidth}
	\centering    
    \begin{tikzpicture}
        \begin{axis}[
            ymode=log,
            xmin = 0, xmax = 125,
            ymin = 1e-6,
            ymax = 1e2,
            ytick = {1e2, 1e0, 1e-2, 1e-4, 1e-6},
            grid = both,
            minor tick num = 1,
            major grid style = {lightgray},
            minor grid style = {lightgray!25},
            width = 0.9\textwidth,
            height = 0.45\textwidth,
            legend cell align = {left},
            legend style={at={(0.5, 0.95)},anchor= north},
            legend columns = 5,
            xlabel = {size $n$ of the reduced basis},
            ylabel = {$\epsilon_{\max,\text{rel}, |\cdot|_{1;\Omega}}(n)$}
        ]
            \addplot[mark=none, gray, samples=2,domain=0:300, style= dashed, forget plot]
            {1e-5}
            node[pos=0.04, above]{$\epsilon_{tol}$};

            \addplot[matlab4, style= thick] table [x = {n}, y = {err}] {\datarev};
            \addlegendentry{CWG}
            
            \addplot[matlab9, style= thick] table [x = {n}, y = {err}] {\datamone};
            \addlegendentry{$\lambda=1$}
            
            \addplot[matlab2, style= thick] table [x = {n}, y = {err}] {\datamtwo};
            \addlegendentry{$\lambda=0.03$}
            
            \addplot[matlab3, style= thick] table [x = {n}, y = {err}] {\datamthree};
            \addlegendentry{$\lambda=10^{-3}$}

            \addplot[matlab5, style= thick] table [x = {n}, y = {err}] {\datamfive};
            \addlegendentry{$\lambda=10^{-5}$}
        \end{axis}
    \end{tikzpicture}
    \subcaption{$2\times 3$ thermal block, bulk criterion}
    \label{fig:err_decay23}
\end{subfigure}
    \caption{Error decay in $\|\cdot\|_{1;\Omega}$ for different bulk parameters $\lambda$ in comparison with the classical weak greedy (CWG) method}.
    \label{fig:err_decay}
\end{figure}

\begin{figure}[!htb]
    \pgfplotstableread{Data/2x2/thermalblock_lambda_cwg.data}{\datarev}
    \pgfplotstableread{Data/2x2/thermalblock_pod_lambda1.0.data}{\datazero}
    \pgfplotstableread{Data/2x2/thermalblock_pod_lambda0.1.data}{\datamone}
    \pgfplotstableread{Data/2x2/thermalblock_pod_lambda0.03.data}{\datamtwo}
    \pgfplotstableread{Data/2x2/thermalblock_pod_lambda0.001.data}{\datamthree}
    \pgfplotstableread{Data/2x2/thermalblock_pod_lambda0.0001.data}{\datamfour}
    \pgfplotstableread{Data/2x2/thermalblock_pod_lambda1e-05.data}{\datamfive}
    \def\vCone{2.0260e5}
    \def\vcone{4.0823}
    \def\valpha{5.74208e-1}
    \def\voffsetdown{1e-2}
    \def\voffsetup{5e3}
\begin{subfigure}[c]{\textwidth}
    \centering
    \begin{tikzpicture}
        \begin{axis}[
            ymode=log,
            xmin = 0, xmax = 26,
            ymin = 8e-7,
            ymax = 1e2,
            ytick = {1e0, 1e-3, 1e-6, 1e-9, 1e-12},
            grid = both,
            minor tick num = 1,
            major grid style = {lightgray},
            minor grid style = {lightgray!25},
            width = 0.9\textwidth,
            height = 0.45\textwidth,
            legend cell align = {left},
            legend style={at={(0.5, 0.95)},anchor= north},
            legend columns = 5,
            xlabel = {size $n$ of the reduced basis},
            ylabel = {$\epsilon_{\max,\text{rel}, |\cdot|_{1;\Omega}}(n)$}
        ]
            \addplot[mark=none, gray, samples=2,domain=0:50, style= dashed, forget plot,thick]
            {1e-5}
            node[pos=0.05, above]{$\epsilon_{tol}$};

            \addplot[matlab4, style= thick, mark = o] table [x = {n}, y = {err}] {\datarev};
            \addlegendentry{CWG}
            
            \addplot[matlab9, style= thick, mark = star] table [x = {n}, y = {err}] {\datazero};
            \addlegendentry{$\lambda=1$}
            
            \addplot[matlab2, style= thick, mark = square] table [x = {n}, y = {err}] {\datamtwo};
            \addlegendentry{$\lambda=0.03$}
            
            \addplot[matlab3, style= thick, mark = diamond] table [x = {n}, y = {err}] {\datamthree};
            \addlegendentry{$\lambda=10^{-3}$}
            
            \addplot[matlab5, style= thick, mark = pentagon] table [x = {n}, y = {err}] {\datamfive};
            \addlegendentry{$\lambda=10^{-5}$}
        \end{axis}
    \end{tikzpicture}
    \subcaption{$2\times 2$ thermal block, POD variant}
\end{subfigure}
    \pgfplotstableread{Data/2x3/thermalblock_lambda_cwg.data}{\datarev}
    \pgfplotstableread{Data/2x3/thermalblock_pod_lambda1.0.data}{\datazero}
    \pgfplotstableread{Data/2x3/thermalblock_pod_lambda0.1.data}{\datamone}
    \pgfplotstableread{Data/2x3/thermalblock_pod_lambda0.03.data}{\datamtwo}
    \pgfplotstableread{Data/2x3/thermalblock_pod_lambda0.001.data}{\datamthree}
    \pgfplotstableread{Data/2x3/thermalblock_pod_lambda0.0001.data}{\datamfour}
    \pgfplotstableread{Data/2x3/thermalblock_pod_lambda1e-05.data}{\datamfive}
    \def\vCone{0.1089}
    \def\vcone{0.0499}
    \def\valpha{0.9906}
    \def\voffsetdown{3e-1}
    \def\voffsetup{6e1}
\begin{subfigure}[c]{\textwidth}
	\centering    
    \begin{tikzpicture}
        \begin{axis}[
            ymode=log,
            xmin = 0, xmax = 125,
            ymin = 1e-6,
            ymax = 1e2,
            ytick = {1e2, 1e0, 1e-2, 1e-4, 1e-6},
            grid = both,
            minor tick num = 1,
            major grid style = {lightgray},
            minor grid style = {lightgray!25},
            width = 0.9\textwidth,
            height = 0.45\textwidth,
            legend cell align = {left},
            legend style={at={(0.5, 0.95)},anchor= north},
            legend columns = 5,
            xlabel = {size $n$ of the reduced basis},
            ylabel = {$\epsilon_{\max,\text{rel}, |\cdot|_{1;\Omega}}(n)$}
        ]
            \addplot[mark=none, gray, samples=2,domain=0:300, style= dashed, forget plot]
            {1e-5}
            node[pos=0.04, above]{$\epsilon_{tol}$};

            \addplot[matlab4, style= thick] table [x = {n}, y = {err}] {\datarev};
            \addlegendentry{CWG}
            
            \addplot[matlab9, style= thick] table [x = {n}, y = {err}] {\datamone};
            \addlegendentry{$\lambda=1$}
            
            \addplot[matlab2, style= thick] table [x = {n}, y = {err}] {\datamtwo};
            \addlegendentry{$\lambda=0.03$}
            
            \addplot[matlab3, style= thick] table [x = {n}, y = {err}] {\datamthree};
            \addlegendentry{$\lambda=10^{-3}$}

            \addplot[matlab5, style= thick] table [x = {n}, y = {err}] {\datamfive};
            \addlegendentry{$\lambda=10^{-5}$}

        \end{axis}
    \end{tikzpicture}
    \subcaption{$2\times 3$ thermal block, POD variant.}
\end{subfigure}
    \caption{Error decay in $\|\cdot\|_{1;\Omega}$ for different POD parameters.}
    \label{fig:err_decay_pod}
\end{figure}

We summarize some observations in Figures \ref{fig:err_decay} and \ref{fig:err_decay_pod}:
\begin{itemize}
	\item The results for the bulk criterion and the POD variant are very similar. 
	\item The \emph{asymptotic} decay of the batch greedy is very similar to the one of the classical weak greedy (CWG) for all choices of $\lambda$. This confirms our theoretical findings, namely that the asymptotic rate of the greedy algorithm is preserved.
	\item For small values of $\lambda$ we observe \enquote{plateaus}, where the error stagnates with increasing $n$. 
	\item The error is not strictly increasing for increasing values of $\lambda$.
	\item For the $2 \times 2$ thermal block and the bulk criterion batch greedy, the error decay for CWG and $\lambda=1$ coincide. For the POD variant and the $2 \times 3$ block, the decays are very similar. 
	\item The error decay is less uniform for the $2\times 2$ block and small values of $\lambda$ as compared to the $2\times 3$ block. We expect that this effect is more pronounced for more challenging problems (here, larger values of $p_x$ and $p_y$).
\end{itemize}

\subsection{Batch selection}
As mentioned in \S\ref{Sec:bminusone}, the sample values determining the snapshots in the batch can be varied. As long as the first snapshots is chosen by maximizing the error estimate, our above analysis remains valid. Hence, we compare four variants: (1) the CWG, (2) maximizing the error estimator $b$ times, (3) the first one corresponding to maximizing the error estimate and choosing the remaining $b-1$ at random and (4) choose all $b$ at random (where our analysis fails). The results are shown in Figure \ref{fig:err_decay_random} for the $2\times 3$ thermal block with the bulk variant (top, Fig.\ \ref{fig:err_decay_23_bulk_random}) and the POD version (bottom, Fig.\ \ref{fig:err_decay_23_pod_random}), where we depict the error versus the size of the reduced basis. 
We observe asymptotically very similar results. Only the fully random variant seems to be a bit worse. It seems that at least for the test problems considered here, the choice of the batch samples does not have a strong influence on the greedy convergence rate.
\begin{figure}[!htb]
\pgfplotstableread{Data/2x3/randombatch_lambda0.01_bulk.data}{\data}
\begin{subfigure}[c]{.9\textwidth}
    \centering
    \begin{tikzpicture}
        \begin{axis}[
            ymode=log,
            xmin = 0, xmax = 120,
            ytick = {1e0, 1e-3, 1e-6, 1e-9, 1e-12},
            grid = both,
            minor tick num = 1,
            major grid style = {lightgray},
            minor grid style = {lightgray!25},
            width = 0.9\textwidth,
            height = 0.45\textwidth,
            legend cell align = {left},
            legend pos = north east,
            xlabel = {size $n$ of the reduced basis},
            ylabel = {$\epsilon_{\max,\text{rel}, |\cdot|_{1;\Omega}}(n)$}
        ]
            \addplot[mark=none, gray, samples=2,domain=0:150, style= dashed, forget plot,thick]
            {1e-5}
            node[pos=0.07, above]{$\epsilon_{tol}$};

            \addplot[matlab4, style= thick] table [x = {n}, y = {cwg}] {\data};
            \addlegendentry{CWG}
            
            \addplot[matlab9, style= thick] table [x = {n}, y = {standard}] {\data};
            \addlegendentry{$b$ max}
            
            \addplot[matlab2, style= thick] table [x = {n}, y = {random}] {\data};
            \addlegendentry{random $b-1$}
            
            \addplot[matlab3, style= thick] table [x = {n}, y = {allrandom}] {\data};
            \addlegendentry{allrandom}
            
        \end{axis}
    \end{tikzpicture}
    \subcaption{bulk variant}
    \label{fig:err_decay_23_bulk_random}
\end{subfigure}
    \pgfplotstableread{Data/2x3/randombatch_lambda0.01_pod.data}{\data}
\begin{subfigure}[c]{.9\textwidth}
	\centering    
    \begin{tikzpicture}
        \begin{axis}[
            ymode=log,
            xmin = 0, xmax = 120,
            ytick = {1e2, 1e0, 1e-2, 1e-4, 1e-6},
            grid = both,
            minor tick num = 1,
            major grid style = {lightgray},
            minor grid style = {lightgray!25},
            width = 0.9\textwidth,
            height = 0.45\textwidth,
            legend cell align = {left},
            legend pos = north east,
            legend columns = 1,
            xlabel = {size $n$ of the reduced basis},
            ylabel = {$\epsilon_{\max,\text{rel}, |\cdot|_{1;\Omega}}(n)$}
        ]
            \addplot[mark=none, gray, samples=2,domain=0:150, style= dashed, forget plot]
            {1e-5}
            node[pos=0.07, above]{$\epsilon_{tol}$};

            \addplot[matlab4, style= thick] table [x = {n}, y = {cwg}] {\data};
            \addlegendentry{CWG}
            
            \addplot[matlab9, style= thick] table [x = {n}, y = {standard}] {\data};
            \addlegendentry{$b$ max}
            
            \addplot[matlab2, style= thick] table [x = {n}, y = {random}] {\data};
            \addlegendentry{random $b-1$}
            
            \addplot[matlab3, style= thick] table [x = {n}, y = {allrandom}] {\data};
            \addlegendentry{allrandom}
        \end{axis}
    \end{tikzpicture}
    \subcaption{POD variant}
    \label{fig:err_decay_23_pod_random}
\end{subfigure}
    \caption{Error decay in $\|\cdot\|_{1;\Omega}$ for different batch selection mechanisms in comparison with the classical weak greedy (CWG) method for the $2 \times 3$ thermal block with $\lambda=0.01$.}
    \label{fig:err_decay_random}
\end{figure}

\subsection{Dimension of the reduced systems}
In the next experiment, we fix the target relative training tolerance as $10^{-5}$ and compare the \enquote{final dimension} of the reduced system (to achieve the tolerance) determined by the batch greedy method for different values of $\lambda$. In order to compare with the CWG, we display the \emph{ratio} of the dimension of the reduced system by the batch and classical weak greedy. Moreover, we monitor the \enquote{effective batch size}, by which we denote the average number of snapshots that are added to the reduced basis in each iteration (hence, it can take any value between 1 and the used batch size, here $b=30$). The results are depicted in Figure \ref{fig:basis_batch}, both for the bulk criterion and POD variant.

We observe that the final dimension (graph on the top) of the reduced system increases as $\lambda$ decreases (which means that online computing times and storage demands will grow), but the growth is moderate for reasonable choices of $\lambda$. The effective batch size (graph on the bottom) also increases for decreasing $\lambda$, which is to be expected. For the $2 \times 2$ block, up to $30\%$ of the batch elements are added to the reduced basis, whereas roughly $50\%$ are used for the $2 \times 3$ thermal block. Again, we conclude that tougher problems can potentially benefit from the batch version.
\begin{figure}[!htb]
    \pgfplotstableread{Data/2x2/thermalblock_lambda_overall.data}{\datatwotwo}
    \pgfplotstableread{Data/2x3/thermalblock_lambda_overall.data}{\datatwothree}
    \pgfplotstableread{Data/2x2/thermalblock_pod_lambda_overall.data}{\datatwotwopod}
    \pgfplotstableread{Data/2x3/thermalblock_pod_lambda_overall.data}{\datatwothreepod}
    \begin{subfigure}[c]{\textwidth}
    \centering
    \begin{tikzpicture}
        \begin{axis}[
            xmode=log,
            xmin = 1e-5, xmax = 1,
            grid = both,
            minor tick num = 1,
            major grid style = {lightgray},
            minor grid style = {lightgray!25},
            width = 0.9\textwidth,
            height = 0.45\textwidth,
            legend cell align = {left},
            legend pos = {north east},
            xlabel = {parameter $\lambda$ (bulk/POD)},
            ylabel = {final basis size ratio}
        ]
            \addplot[matlab4, mark = oplus, style= very thick, mark size = 2pt,
             y filter/.expression={\thisrow{lambda}==2.0 ? nan : y},
             unbounded coords=jump] table [x = {lambda}, y = {rel_size}] {\datatwotwo};

            \addplot[matlab9, mark = triangle, style= very thick, mark size = 2pt,
             y filter/.expression={\thisrow{lambda}==2.0 ? nan : y},
             unbounded coords=jump] table [x = {lambda}, y = {rel_size}] {\datatwothree};
    
			\addplot[matlab2, mark = oplus, style= very thick, mark size = 2pt,
             y filter/.expression={\thisrow{lambda}==2.0 ? nan : y},
             unbounded coords=jump] table [x = {lambda}, y = {rel_size}] {\datatwotwopod};

            \addplot[matlab3, mark = triangle, style= very thick, mark size = 2pt,
             y filter/.expression={\thisrow{lambda}==2.0 ? nan : y},
             unbounded coords=jump] table [x = {lambda}, y = {rel_size}] {\datatwothreepod};

            \legend{
                $2\times 2$ (bulk), $2\times 3$ (bulk),
                $2\times 2$ (POD), $2\times 3$ (POD)
            }
        \end{axis}
    \end{tikzpicture}
\end{subfigure}
 \begin{subfigure}[c]{\textwidth}
 	\centering
    \begin{tikzpicture}
        \begin{axis}[
            xmode=log,
            xmin = 1e-5, xmax = 1,
            ymin = 1, ymax = 15,
            grid = both,
            minor tick num = 1,
            major grid style = {lightgray},
            minor grid style = {lightgray!25},
            width = 0.9\textwidth,
            height = 0.45\textwidth,
            legend cell align = {left},
            legend pos = {north east},
            xlabel = {parameter $\lambda$ (bulk/POD)},
            ylabel = {effective batch size}
        ]
            \addplot[matlab4, mark = oplus, style= very thick, mark size = 2pt,
             y filter/.expression={\thisrow{lambda}==2.0 ? nan : y},
             unbounded coords=jump] table [x = {lambda}, y = {eff_bs}] {\datatwotwo};
            
            \addplot[matlab9, mark = triangle, style= very thick, mark size = 2pt,
             y filter/.expression={\thisrow{lambda}==2.0 ? nan : y},
             unbounded coords=jump] table [x = {lambda}, y = {eff_bs}] {\datatwothree};
            
              \addplot[matlab2, mark = oplus, style= very thick, mark size = 2pt,
             y filter/.expression={\thisrow{lambda}==2.0 ? nan : y},
             unbounded coords=jump] table [x = {lambda}, y = {eff_bs}] {\datatwotwopod};
            
            \addplot[matlab3, mark = triangle, style= very thick, mark size = 2pt,
             y filter/.expression={\thisrow{lambda}==2.0 ? nan : y},
             unbounded coords=jump] table [x = {lambda}, y = {eff_bs}] {\datatwothreepod};

            \legend{
                 $2\times 2$ (bulk), $2\times 3$ (bulk),
                 $2\times 2$ (POD), $2\times 3$ (POD)
            }
        \end{axis}
    \end{tikzpicture}
    \end{subfigure}
    \caption{Relative change in the final basis size (compared to the classical weak greedy) and the effective batch size for different bulk parameters.}\label{fig:basis_batch}
\end{figure}

\subsection{Offline \& online computing times}
Our next aim is the investigation of the computational times since the ultimate aim of the introduction of a batch is to reduce computing times in particular in the offline stage.

In Figure \ref{fig:times}, we show computing times (CPU) for the offline and the online phase (for both variants, bulk criterion and POD). Again, we normalize the CPU times w.r.t.\ the classical greedy, namely, we indicate the ratio of the CPU time for a given $\lambda$ and the time for the standard greedy. For determining the online time, we calculated the average computing time of a reduced solution for $500$ randomly chosen parameters. Again, we show results for the bulk criterion and the POD version.

We note that the \emph{online} CPU times resemble the basis size shown in Figure~\ref{fig:basis_batch}, which is to be expected. Moreover, the online times for the two variants are very close, which is also no surprise as the online CPU time is mainly determined by the size of the reduced model. Finally, we observe a quite moderate increase of the online time due to the introduction of a batch.

Concerning the \emph{offline} CPU times, we clearly see the positive impact by using parallel batch computations. For the $2\times 2$ block, the offline time decreases with decreasing $\lambda$. However, from a certain value on (here around $10^{-2}$), a further reduction of $\lambda$ leads to no further reduction in the offline time. With an optimal choice of $\lambda$, the offline CPU time is reduced to less than half of the classical approach. A similar observation is made for the POD variant for the $2\times 3$ block. For the bulk criterion case, though, the situation is not so clear. We do see only a very moderate reduction of the offline CPU times for most values of $\lambda$, which is counterintuitive at a first glance. To this end, we will investigate the split-up of the CPU time in more detail next.

\begin{figure}[!htb]
    \pgfplotstableread{Data/2x2/thermalblock_lambda_overall.data}{\data}
    \pgfplotstableread{Data/2x2/thermalblock_pod_lambda_overall.data}{\datapod}
\begin{subfigure}[c]{\textwidth}
	\centering
    \begin{tikzpicture}
        \begin{axis}[
            xmode=log,
            xmin = 1e-5, xmax = 1,
            grid = both,
            minor tick num = 1,
            major grid style = {lightgray},
            minor grid style = {lightgray!25},
            ytick = {0.25, 0.5, 0.75, 1, 1.25},
            width = 0.9\textwidth,
            height = 0.45\textwidth,
            legend cell align = {left},
            legend style={at={(0.05, 0.7)},anchor= north west},
            xlabel = {parameter $\lambda$ (bulk/POD)},
            ylabel = {CPU ratio}
        ]
            \addplot[matlab4, mark = *, style= very thick, mark size = 2pt,
            y filter/.expression={\thisrow{lambda}==2.0 ? nan : y}] table [x = {lambda}, y = {t_online_n}] {\data};
            \addlegendentry{online (bulk)}
            \addplot[matlab9, mark = triangle, style= very thick, mark size = 2pt,
            y filter/.expression={\thisrow{lambda}==2.0 ? nan : y}] table [x = {lambda}, y = {t_offline_n}] {\data};
            \addlegendentry{offline (bulk)}
            \addplot[matlab2, mark = *, style= very thick, mark size = 2pt,
            y filter/.expression={\thisrow{lambda}==2.0 ? nan : y}] table [x = {lambda}, y = {t_online_n}] {\datapod};
            \addlegendentry{online (POD)}
            \addplot[matlab3, mark = triangle, style= very thick, mark size = 2pt,
            y filter/.expression={\thisrow{lambda}==2.0 ? nan : y}] table [x = {lambda}, y = {t_offline_n}] {\datapod};
            \addlegendentry{offline (POD)}
            \addplot[mark=none, black, samples=2,domain=1e-6:1e3, style=dashed] {1};

            \addplot[mark=none, black, samples=2,domain=1e-6:1e3, style=dashed] {1};
        \end{axis}
    \end{tikzpicture}
    \caption{$2\times 2$ thermal block} \label{fig:times22}
\end{subfigure}
    \pgfplotstableread{Data/2x3/thermalblock_lambda_overall.data}{\data}
    \pgfplotstableread{Data/2x3/thermalblock_pod_lambda_overall.data}{\datapod}
\begin{subfigure}[c]{\textwidth}
	\centering
    \begin{tikzpicture}
        \begin{axis}[
            xmode=log,
            xmin = 1e-5, xmax = 1,
            ymax = 1.5,
            grid = both,
            minor tick num = 1,
            ytick = {0.5, 0.75, 1, 1.25, 1.5},
            major grid style = {lightgray},
            minor grid style = {lightgray!25},
            width = 0.9\textwidth,
            height = 0.45\textwidth,
            legend cell align = {left},
            legend pos = south west,
            xlabel = {parameter $\lambda$ (bulk/POD)},
            ylabel = {CPU ratio}
        ]
            \addplot[matlab4, mark = *, style= very thick, mark size = 2pt,
            y filter/.expression={\thisrow{lambda}==2.0 ? nan : y}] table [x = {lambda}, y = {t_online_n}] {\data};
            \addlegendentry{online (bulk)}
            \addplot[matlab9, mark = triangle, style= very thick, mark size = 2pt,
            y filter/.expression={\thisrow{lambda}==2.0 ? nan : y}] table [x = {lambda}, y = {t_offline_n}] {\data};
            \addlegendentry{offline (bulk)}
            \addplot[matlab2, mark = *, style= very thick, mark size = 2pt,
            y filter/.expression={\thisrow{lambda}==2.0 ? nan : y}] table [x = {lambda}, y = {t_online_n}] {\datapod};
            \addlegendentry{online (POD)}
            \addplot[matlab3, mark = triangle, style= very thick, mark size = 2pt,
            y filter/.expression={\thisrow{lambda}==2.0 ? nan : y}] table [x = {lambda}, y = {t_offline_n}] {\datapod};
            \addlegendentry{offline (POD)}
            \addplot[mark=none, black, samples=2,domain=1e-6:1e2, style=dashed] {1};

            \addplot[mark=none, black, samples=2,domain=1e-6:1e2, style=dashed] {1};
        \end{axis}
    \end{tikzpicture}
    \caption{$2\times 3$ thermal block} \label{fig:times33}
    \end{subfigure}
    \caption{CPU time ratio of the batch versus the classical greedy method for different parameters $\lambda$ (bulk, POD), $2\times 2$ (top) and $2\times 3$ thermal block.} \label{fig:times}
\end{figure}

\subsection{CPU times in detail} 
As announced, we investigate the offline CPU times a bit more in detail and show a split-up of the offline phase into the parts, namely
\begin{itemize}
\item \emph{Solve}: using the full order model for the snapshot computation;
\item \emph{Evaluate}: evaluation of the error estimator on the training set;
\item \emph{Extend}: extension of the reduced basis including the orthogonalization;
\item \emph{Reduce}: Update of the reduced model, error estimator and projectors.
\end{itemize}
The results are shown in Figure \ref{fig:split}, for both variants (POD is shown in transparent mode). 
Clearly, the parallel computation reduces the \emph{Solve} part as $\lambda$ decreases. For the $2\times 2$ block case, also the overall offline time is significantly reduced. The same holds true for the $2\times 3$ block using the POD variant. The situation is different for the bulk criterion variant. As we see, the time shares for \emph{Evaluate} and \emph{Reduce} are significant. This comes from the fact that adding one snapshot at a time by checking the bulk criterion requires to update the reduced model. The POD variant obviously performs much better. 
For the special case $\lambda=0$ in the bulk version, we add all snapshots at once and do not need to update the reduced model after every snapshot (rather at the end of the iteration). This should reduce the shares of \emph{Evaluate} and \emph{Reduce}. However, even though we used a smaller batch size $b=10$ (instead of $b=30$) for this special case, this advantage seems to be very small due to the poor convergence. For $\lambda=0$ and $b=30$ (not displayed here) the offline times for both test problems were significantly higher than the offline times of the CWG.

\begin{figure}[!htb]
    \pgfplotstableread{Data/2x2/thermalblock_lambda_overall.data}{\data}
    \pgfplotstableread{Data/2x2/thermalblock_pod_lambda_overall.data}{\datapod}
\begin{subfigure}[c]{\textwidth}
	\centering
    \begin{tikzpicture}
        \begin{axis}[
            ybar stacked,
            xmin = -0.75, xmax = 12.75,
            ymin = 0,
            ymajorgrids = true,
            xmajorgrids = false,
            width = 0.9\textwidth,
            height = 0.5\textwidth,
            legend cell align = {left},
            legend pos = north west,
            legend columns = 2,
            xlabel = {parameter $\lambda$ (left: bulk, right: POD)},
            ylabel = {offline time $[s]$},
            xtick=data,     
            xticklabels = {0, 1e-5, 3e-5, 1e-4, 3e-4, 1e-3, 3e-3, 0.01, 0.03, 0.1, 0.3, 1, CWG},
            scale ticks above exponent=2,
        ]
        \addplot[fill=matlab4,bar width=0.3] table [y = t_solve, x expr=\coordindex-0.2] {\data};
        \addplot[fill=matlab9,bar width=0.3] table [y = t_evaluate, x expr=\coordindex-0.2] {\data};
        \addplot[fill=matlab2,bar width=0.3] table [y = t_extend, x expr=\coordindex-0.2] {\data};
        \addplot[fill=matlab3,bar width=0.3] table [y = t_reduce, x expr=\coordindex-0.2] {\data};
        \addplot[fill=mygray,bar width=0.3] table [y = t_other, x expr=\coordindex-0.2] {\data};
        
        \legend{
            Solve, Evaluate, Extend, Reduce, Other
        }
        \end{axis}
      \begin{axis}[
            ybar stacked,
            xmin = -0.75, xmax = 12.75,
            ymin = 0,
            width = 0.9\textwidth,
            height = 0.5\textwidth,
            ticks=none,
        ]
        \addplot[fill=matlab4,semitransparent, bar width=0.3] table [y = t_solve, x expr=\coordindex+0.2] {\datapod};
        \addplot[fill=matlab9,semitransparent,bar width=0.3] table [y = t_evaluate, x expr=\coordindex+0.2] {\datapod};
        \addplot[fill=matlab2,semitransparent,bar width=0.3] table [y = t_extend, x expr=\coordindex+0.2] {\datapod};
        \addplot[fill=matlab3,semitransparent,bar width=0.3] table [y = t_reduce, x expr=\coordindex+0.2] {\datapod};
        \addplot[fill=mygray,semitransparent,bar width=0.3] table [y expr= \thisrow{t_other}+\thisrow{t_pod}, x expr=\coordindex+0.2] {\datapod};
        
        \end{axis}
    \end{tikzpicture}
    \caption{$2 \times 2$ thermal block}\label{fig:split22}
\end{subfigure}
\pgfplotstableread{Data/2x3/thermalblock_lambda_overall.data}{\data}
\pgfplotstableread{Data/2x3/thermalblock_pod_lambda_overall.data}{\datapod}

\begin{subfigure}[c]{\textwidth}
	\centering
    \begin{tikzpicture}
        \begin{axis}[
            ybar stacked,
            xmin = -0.75, xmax = 12.75,
            ymin = 0,
            ymajorgrids = true,
            xmajorgrids = false,
            width = 0.9\textwidth,
            height = 0.5\textwidth,
            legend cell align = {left},
            legend style={at={(0.6, 0.97)},anchor= north},
            legend columns = 2,
            xlabel = {parameter $\lambda$ (left: bulk, right: POD)},
            ylabel = {offline time $[s]$},
            xtick=data,     
            xticklabels = {0, 1e-5, 3e-5, 1e-4, 3e-4, 1e-3, 3e-3, 0.01, 0.03, 0.1, 0.3, 1, CWG},
        ]
        \addplot[fill=matlab4,bar width=0.3] table [y = t_solve, x expr=\coordindex-0.2] {\data};
        \addplot[fill=matlab9,bar width=0.3] table [y = t_evaluate, x expr=\coordindex-0.2] {\data};
        \addplot[fill=matlab2,bar width=0.3] table [y = t_extend, x expr=\coordindex-0.2] {\data};
        \addplot[fill=matlab3,bar width=0.3] table [y = t_reduce, x expr=\coordindex-0.2] {\data};
        \addplot[fill=mygray,bar width=0.3] table [y = t_other, x expr=\coordindex-0.2] {\data};
            
        \legend{
            Solve, Evaluate, Extend, Reduce, Other
        }
        \end{axis}
              \begin{axis}[
            ybar stacked,
            xmin = -0.75, xmax = 12.75,
            ymin = 0,
            width = 0.9\textwidth,
            height = 0.5\textwidth,
            ticks=none,
        ]
        \addplot[fill=matlab4,semitransparent, bar width=0.3] table [y = t_solve, x expr=\coordindex+0.2] {\datapod};
        \addplot[fill=matlab9,semitransparent,bar width=0.3] table [y = t_evaluate, x expr=\coordindex+0.2] {\datapod};
        \addplot[fill=matlab2,semitransparent,bar width=0.3] table [y = t_extend, x expr=\coordindex+0.2] {\datapod};
        \addplot[fill=matlab3,semitransparent,bar width=0.3] table [y = t_reduce, x expr=\coordindex+0.2] {\datapod};
        \addplot[fill=mygray,semitransparent,bar width=0.3] table [y expr= \thisrow{t_other}+\thisrow{t_pod}, x expr=\coordindex+0.2] {\datapod};
        
        \end{axis}

    \end{tikzpicture}
    \caption{$2 \times 3$ thermal block}\label{fig:split33}
   \end{subfigure}
    \caption{Offline computing times for different parameters (CWG: classical weak greedy).}\label{fig:split}
\end{figure}

\subsection{The break-even point}
Finally, we consider the influence of the batch to the \enquote{break-even point}, which determines the number of parameter-queries from which a reduced simulation (in\-clud\-ing offline and online computing times) pays off as compared to a repeated call of the full order model. To be precise, the break-even point is $k^*= \left\lceil t_{\text{offline}}/(t_{\text{full}} - t_{\text{online}}) \right\rceil$,  where $t_{\text{offline}}$ denotes the offline time, $t_{\text{full}}$ the time to calculate a full order solution (also for the snapshots), and $t_{\text{online}}$ is the time it takes to compute a reduced solution. In Table \ref{tab:break_even}, we see that for the parallel batch greedy algorithm the break-even points are substantially lower than for the classical variant, which supports the argument that the trade-off between increased online and decreased offline time can be worthwhile. We also see that the numbers are quite similar no matter how the samples for the batch are chosen.
\begin{table}[!htb]
\centering \footnotesize
\begin{tabular}{|c||c|c|c|c|c||c|c|c|c|c|}
\hline
& \multicolumn{5}{c||}{$2 \times 2$} & \multicolumn{5}{c|}{$2 \times 3$}\\ \cline{2-11}
($t$ in $\text{sec}$) & CWG & \multicolumn{2}{c|}{bulk} & \multicolumn{2}{c||}{POD} & CWG & \multicolumn{2}{c|}{bulk} & \multicolumn{2}{c|}{POD}\\
\hline
\multirow{2}*{select.} & \multirow{2}*{--} & \multirow{2}*{batch} & rand. & \multirow{2}*{batch} & rand. & \multirow{2}*{--} & \multirow{2}*{batch} & rand. & \multirow{2}*{batch} & rand.\\
 &  &  & $b-1$ & & $b-1$ & & & $b-1$ & & $b-1$\\
\hline & & & & & & & & & &\\[-1em]
$\lambda$ & -- & $10^{-3}$ & $10^{-3}$ & $3 \cdot 10^{-5}$ & $10^{-3}$ & -- & $10^{-2}$ & $10^{-2}$ & $ 10^{-4}$ & $10^{-2}$\\
\hline
$t_{\text{full}}$ & 77.38 & 77.38 & 77.38 & 77.38 & 77.38 & 77.38 & 77.38 & 77.38 & 77.38 & 77.38\\
$t_{\text{offline}}$ & 3163 & 1131 & 1028 & 928 & 941 & 18612 & 11631 & 12154 & 3313 & 2580\\
$t_{\text{online}}$ & 0.0342 & 0.0372 & 0.0369 & 0.0366 & 0.0394 & 0.136 & 0.141 & 0.148 & 0.146 & 0.148\\
\hline
$k^* $ & 41 & 15 & 14 & 12 & 13 & 241 & 151 & 158 & 43 & 34\\
\hline
\end{tabular}
\caption{Break-even points $k^*$ for the reduced models of the classical greedy compared to the two variants of the batch greedy algorithm for some values of $\lambda$.}\label{tab:break_even}
\end{table}

\begin{remark}
	Our above results indicate that it might be interesting to a priori determine an optimal batch size $b$ and bulk parameter $\lambda$ in order to balance offline and online times for a given problem. However, we believe that the choice of batch size should be motivated more by the hardware to be used in order to maximize the usage of all available CPU cores. In most cases it seems meaningful to choose the batch size equal to the number of parallel processes via \emph{MPI}. This is under the assumption $\lambda>0$ (for either variant) and the number of available CPU cores is low double digits. For $\lambda=0$, a more conservative $b\le 10$ seems to work best.

	The optimal choice of $\lambda$ is more challenging. The results presented in this work (and further investigations not presented here) suggest that $\lambda \in [10^{-2},10^{-1}]$ seem to work quite well in all meaningful cases for the bulk version. For the POD version, one can go even smaller with $\lambda \in [10^{-4},10^{-2}]$. A more thorough investigation of an optimal or even adaptive choice for $\lambda$ (and $b$) is a topic for future research.
\hfill$\diamond$
\end{remark}

\subsection{Influence of the training set size}
Finally, we investigate the influence of the size of the training set. While the error analysis in \S \ref{Sec:Analysis} is based upon a possibly \emph{infinite} set $\cF$, a \emph{finite} training set is used in practice. However, a larger training set also contains an increasing number of snapshots that are quite similar. Hence, one could expect that increasingly many snapshots would have to be rejected as the training set is refined.

In order to address this concern, we consider again the $2\times 2$ thermal block and use uniformly refined training sets of size $5^4$, $10^4$, $15^4$, $20^4$ and $25^4$. The resulting effective batch sizes are shown in Figure \ref{fig:snap_eff_bs} for both variants. In fact, we see that the effective batch size is smaller for larger training sets, but the effect is rather moderate.

\begin{figure}[!htb]
    \pgfplotstableread{Data/2x2/thermalblock_snap_lambda1.0_overall.data}{\datalambdazero}
    \pgfplotstableread{Data/2x2/thermalblock_snap_lambda0.1_overall.data}{\datalambdaone}
    \pgfplotstableread{Data/2x2/thermalblock_snap_lambda0.01_overall.data}{\datalambdatwo}
    \pgfplotstableread{Data/2x2/thermalblock_snap_lambda0.001_overall.data}{\datalambdathree}
    \pgfplotstableread{Data/2x2/thermalblock_snap_lambda0.0001_overall.data}{\datalambdafour}
    \pgfplotstableread{Data/2x2/thermalblock_snap_lambda1e-05_overall.data}{\datalambdafive}
    \begin{subfigure}[c]{.47\textwidth}
    \centering
    \begin{tikzpicture}
        \begin{axis}[
            ymin = 0, ymax = 15,
            grid = both,
            minor tick num = 1,
            major grid style = {lightgray},
            minor grid style = {lightgray!25},
            width = 0.9\textwidth,
            height = \textwidth,
            legend cell align = {left},
            legend pos = {outer north east},
            xtick={5,10,15,20,25},
            xlabel = {\# of samples per dim.},
            ylabel = {effective batch size}
        ]
            \addplot[matlab5, mark = |, style= very thick, mark size = 2pt] 
            table [x = {snap}, y = {eff_bs}] {\datalambdafive};

            \addplot[matlab1, mark = o, style= very thick, mark size = 2pt] 
            table [x = {snap}, y = {eff_bs}] {\datalambdafour};

            \addplot[matlab3, mark = diamond, style= very thick, mark size = 2pt] 
            table [x = {snap}, y = {eff_bs}] {\datalambdathree};

            \addplot[matlab2, mark = square, style= very thick, mark size = 2pt] 
            table [x = {snap}, y = {eff_bs}] {\datalambdatwo};

            \addplot[matlab9, mark = triangle, style= very thick, mark size = 2pt] 
            table [x = {snap}, y = {eff_bs}] {\datalambdaone};

            \addplot[matlab4, mark = oplus, style= very thick, mark size = 2pt] 
            table [x = {snap}, y = {eff_bs}] {\datalambdazero};

            \legend{
                $\lambda=10^{-5}$,$\lambda=10^{-4}$,$\lambda=10^{-3}$,
                $\lambda=10^{-2}$,$\lambda=10^{-1}$,$\lambda=1$,
            }
        \end{axis}
    \end{tikzpicture}
    \caption{bulk variant.} 
\end{subfigure}
\hspace*{.01cm}
\pgfplotstableread{Data/2x2/thermalblock_pod_snap_lambda1.0_overall.data}{\datalambdazero}
\pgfplotstableread{Data/2x2/thermalblock_pod_snap_lambda0.1_overall.data}{\datalambdaone}
\pgfplotstableread{Data/2x2/thermalblock_pod_snap_lambda0.01_overall.data}{\datalambdatwo}
\pgfplotstableread{Data/2x2/thermalblock_pod_snap_lambda0.001_overall.data}{\datalambdathree}
\pgfplotstableread{Data/2x2/thermalblock_pod_snap_lambda0.0001_overall.data}{\datalambdafour}
\pgfplotstableread{Data/2x2/thermalblock_pod_snap_lambda1e-05_overall.data}{\datalambdafive}
\begin{subfigure}[c]{.47\textwidth}
    \centering
    \begin{tikzpicture}
        \begin{axis}[
            ymin = 0, ymax = 15,
            grid = both,
            minor tick num = 1,
            major grid style = {lightgray},
            minor grid style = {lightgray!25},
            width = 0.9\textwidth,
            height = \textwidth,
            legend cell align = {left},
            legend pos = {north east},
            legend columns = 2,
            xtick={5,10,15,20,25},
            xlabel = {\# of samples per dim.},
            yticklabel pos=right
        ]
            \addplot[matlab4, mark = oplus, style= very thick, mark size = 2pt] 
            table [x = {snap}, y = {eff_bs}] {\datalambdazero};

            \addplot[matlab9, mark = triangle, style= very thick, mark size = 2pt] 
            table [x = {snap}, y = {eff_bs}] {\datalambdaone};
    
            \addplot[matlab2, mark = square, style= very thick, mark size = 2pt] 
            table [x = {snap}, y = {eff_bs}] {\datalambdatwo};

            \addplot[matlab3, mark = diamond, style= very thick, mark size = 2pt] 
            table [x = {snap}, y = {eff_bs}] {\datalambdathree};

            \addplot[matlab1, mark = o, style= very thick, mark size = 2pt] 
            table [x = {snap}, y = {eff_bs}] {\datalambdafour};

            \addplot[matlab5, mark = |, style= very thick, mark size = 2pt] 
            table [x = {snap}, y = {eff_bs}] {\datalambdafive};

        \end{axis}
    \end{tikzpicture}
    \caption{POD variant} 
\end{subfigure}
\caption{Effective batch size for different values of $\lambda$ depending on the number of samples per dimension in the parameter space $\cP$.}\label{fig:snap_eff_bs}
\end{figure}

\subsection{Discussion}
We summarize the findings of our numerical experiments. We have shown the potential of the batch greedy method to speedup the offline phase. While the asymptotic rate of convergence of the classical weak greedy method is preserved in many cases, we observed a speedup of the offline phase by a factor of $30\%$, which is expected to be even more pronounced for more challenging problems. In our experiments, the POD variant turned out to be more efficient compared to the bulk criterion. The selection of the batch as well as the size of the training set seem to have only minor influence on the performance of the method. The parameter $\lambda$ has significant influence on the efficiency and must be chosen problem-dependent. 

\appendix
\section{Error analysis in Banach spaces}
\label{Sec:Appendix_Banach}

We briefly report on the generalization of the results in \S \ref{Sec:Analysis} for Algorithm \ref{alg:wbg} to Banach spaces analogous to \cite[\S 4]{weakgreedy2}. As in the Hilbert space case this amounts minor adaptions, which are highlighted in \hue{blue}. 

Again, we will view the results of the weak batch greedy method as a lower triangular matrix $A := (a_{i,j})_{i,j\in\N_0}$ (see \eqref{eq:A}) to make use of Lemma \ref{lemma:main}. Therefore, we will use the abbreviations $\sigma_n :=  \sigma_n (\cF)_X$ and $d_n := d_n(\cF)_X$ as before, where $X$ is now a Banach space with norm $\norm{\cdot} \equiv \norm{\cdot}_X$. For each $j=0,1,\ldots$ we denote by $\phi_j \in X^*$ the linear functional of norm one that satisfies
\begin{align*}
    \text{(i)} \; \phi_j (V_j)=0, \qquad 
    \text{(ii)} \; \phi_j(f_j) = \dist(f_j, V_j)_X,
\end{align*}
where the existence of such $\phi_j$ is a consequence of the Hahn-Banach theorem (see \cite[Cor.\ IV.14.13]{hewitt2012}). The entries of $A$ are then given by
\begin{align}
\label{eq:entries_A_B}
    a_{i,j} := \phi_j(f_i).
\end{align}
We have similar properties of $A$ as before.

\begin{lemma}\label{La:Psbg_B}
	For Algorithm \ref{alg:wbg} and $A$ in \eqref{eq:entries_A_B}, we have for all $n,m\in\N_0$ with $m>n$  
\begin{enumerate}[\bfseries ({P}1)]
    \item the diagonal elements satisfy $\hue{\lambda} \gamma \sigma_{n} \leq \abs{a_{n,n}} \,\leq\, \sigma_n$;
    \item the entries in the lower triangular half satisfy $\abs{a_{m,n}} \leq \sigma_n$.
\end{enumerate}
\end{lemma}
\begin{proof}
The upper inequality of \emph{\textbf{(P1)}} is easy to see since $a_{n,n} = \dist(f_n,V_n)_X = \sigma_n(f_n) \hue{\leq} \max_{f \in \cF} \dist(f,V_n) = \sigma_n (\cF) = \sigma_n$. 
The lower inequality in \emph{\textbf{(P1)}} directly follows from Corollary \ref{col:algprop} \eqref{prop:batchquality} since $\hue{\lambda} \gamma \sigma_n \leq \sigma_n (f_n) = \abs{a_{n,n}}$. For \emph{\textbf{(P2)}} we have with $n<m$ the bound $\abs{a_{m,n}} = \phi_n(f_m) = \phi_n(f_m - g) \leq \norm{\phi_n}_{X^*} \norm{f_m - g} = \norm{f_m - g}$ for every $g \in V_n$ since $\phi_n(V_n)=0$. Therefore, $\abs{a_{m,n}} \leq \dist(f_m,V_n) = \sigma_n(f_m) \leq \sigma_n$, which concludes the proof.
\end{proof}

This allows us to generalize \cite[\S 4]{weakgreedy2} to the batch greedy algorithm with bulk criterion by substituting $\gamma$ with $\hue{\lambda} \gamma$. It is then straightforward to get the following result as an adaption of \cite[Cor.\ 4.2]{weakgreedy2}.

\begin{corollary}
    For Algorithm \ref{alg:wbg} in Banach spaces with constant $\gamma \in (0,1]$ and bulk parameter $\hue\lambda \in (0,1]$ we have:
    \begin{enumerate}[(i)]
    \item If $d_n(\cF) \leq C_0 n^{-\alpha}$, $n \in \N$ for $\alpha > {0}$, then for any $0 < \beta < \min \left\lbrace \alpha,\frac{1}{2}\right\rbrace$ we have $\sigma_n(\cF) \leq C_1(\hue{\lambda})n^{-\alpha + 1/2 + \beta}$ with
    \begin{align*}
        C_1(\hue{\lambda})
        := \max \Biggl\{ C_0 4^{4\alpha+1} (\hue{\lambda} \gamma)^{-4} \left( \frac{2 \beta +1}{2 \beta} \right)^{\alpha},
        \max_{n=1, \ldots, 7} \left\lbrace n^{\alpha-\beta - 1/2}\right\rbrace \Biggr\}.
    \end{align*}
    \item If $d_n(\cF)\leq C_0 e^{-c_0 n^\alpha}$, $n \in \N$, for $\alpha>0$, then $\sigma_n(\cF) \leq C_1(\hue{\lambda}) \sqrt{n} e^{-c_1(\hue{\lambda}) n^\alpha}$ with 
    \begin{equation*}
        C_1(\hue{\lambda})
        := \max \left \lbrace \frac{\sqrt{2 C_0}}{\hue{\lambda} \gamma},
        1 + \epsilon \right \rbrace
        \quad \text{and} \quad
        c_1(\hue{\lambda}) := \min \left \lbrace 2^{-1-2\alpha}c_0,
        \ln(C_1(\hue{\lambda})) \right \rbrace
    \end{equation*}
    for some (arbitrarily small) $\epsilon > 0$.
    \end{enumerate}
\end{corollary}

Note, that POD is not meaningful in Banach spaces.

\printbibliography

@article{MultiFidelity,
	author = {Feng, Lihong and Lombardi, Luigi and Antonini, Giulio and Benner, Peter},
	title = {Multi-fidelity error estimation accelerates greedy model reduction of complex dynamical systems},
	journal = {Int.\ J.\ Numer.\ Methods Eng.},
	fjournal = {International Journal for Numerical Methods in Engineering},
	volume = {124},
	number = {23},
	pages = {5312-5333},
	year = {2023}
}

@article{OR16,
	author = {Mario Ohlberger and Stephan Rave},
	title = {Reduced Basis Methods: Success, Limitations and Future Challenges},
	journal = {Proceedings of the Conference Algoritmy},
	year = {2016},
	pages = {1--12}
}

@article{buffa-maday-patera-prudhomme-turinici-2012, 
	author={Buffa, Annalisa and Maday, Yvon and Patera, Anthony T. and Prud’homme, Christophe and Turinici, Gabriel}, 
	title={A priori convergence of the Greedy algorithm for the parametrized reduced basis meth\-od}, 
	volume={46}, 
	number={3}, 
	journal={ESAIM: Math. Mod. Numer. Anal.}, 
	publisher={EDP Sciences}, 
	year={2012}, 
	pages={595–603}
}

@article{weakgreedy,
  title={Convergence rates for greedy algorithms in reduced basis methods},
  author={Binev, Peter and Cohen, Albert and Dahmen, Wolfgang and DeVore, Ronald and Petrova, Guergana and Wojtaszczyk, Przemyslaw},
  journal={SIAM J.\ Math.\ Anal.},
  volume={43},
  number={3},
  pages={1457--1472},
  year={2011},
  publisher={SIAM}
}

@book{quarteroni2015reduced,
  title={Reduced basis methods for partial differential equations: an introduction},
  author={Quarteroni, Alfio and Manzoni, Andrea and Negri, Federico},
  volume={92},
  year={2015},
  publisher={Springer},
  address ={Cham},
}

@book{hesthaven2016certified,
  title={Certified reduced basis methods for parametrized partial differential equations},
  author={Hesthaven, Jan S and Rozza, Gianluigi and Stamm, Benjamin},
  volume={590},
  year={2016},
  publisher={Springer},
  address ={Cham},
}

@article{pymor,
	author = {Milk, Ren\'{e} and Rave, Stephan and Schindler, Felix},
	title = {pyMOR -- Generic Algorithms and Interfaces for Model Order Reduction},
	journal = {SIAM J.\ Sci.\ Comp.},
	volume = {38},
	number = {5},
	pages = {194--216},
	year = {2016},
}

@online{pymor2,
    author = {Fritze, Ren\'{e} and Rave, Stephan and Schindler, Felix and Mlinari\'{c}, Petar and Balicki, Linus},
    title  = "{pyMOR}",
    note = {\url{www.pymor.org}. Accessed: 18.07.2024}
}

@article{weakgreedy2,
  title={Greedy algorithms for reduced bases in Banach spaces},
  author={DeVore, Ronald and Petrova, Guergana and Wojtaszczyk, Przemyslaw},
  journal={Constr.\ Approx.},
  volume={37},
  number={3},
  pages={455--466},
  year={2013},
  publisher={Springer}
}

@book{hewitt2012,
  title={Real and abstract analysis: a modern treatment of the theory of functions of a real variable},
  year={1965},
  publisher={Springer Berlin, Heidelberg},
  author={Hewitt, Edwin and Stromberg, Karl}
}

@online{mpi40,
    author = "{Message Passing Interface Forum}",
    title  = "{MPI}: A Message-Passing Interface Standard Version 4.0",
    note    = {\url{www.mpi-forum.org/docs/mpi-4.0/mpi40-report.pdf}. Accessed: 19.04.2024},
}

@ARTICLE{mpi4py,
  author={Dalcin, Lisandro and Fang, Yao-Lung L.},
  journal={Computing in Science \& Engineering}, 
  title={mpi4py: Status Update After 12 Years of De\-vel\-op\-ment}, 
  year={2021},
  volume={23},
  number={4},
  pages={47-54},
}

@Article{numpy,
 title         = {Array programming with {NumPy}},
 author        = {Charles R. Harris and K. Jarrod Millman and St{\'{e}}fan J.
                 van der Walt and Ralf Gommers and Pauli Virtanen and David
                 Cournapeau and Eric Wieser and Julian Taylor and Sebastian
                 Berg and Nathaniel J. Smith and Robert Kern and Matti Picus
                 and Stephan Hoyer and Marten H. van Kerkwijk and Matthew
                 Brett and Allan Haldane and Jaime Fern{\'{a}}ndez del
                 R{\'{i}}o and Mark Wiebe and Pearu Peterson and Pierre
                 G{\'{e}}rard-Marchant and Kevin Sheppard and Tyler Reddy and
                 Warren Weckesser and Hameer Abbasi and Christoph Gohlke and
                 Travis E. Oliphant},
 year          = {2020},
 month         = sep,
 journal       = {Nature},
 volume        = {585},
 number        = {7825},
 pages         = {357--362},
 publisher     = {Springer Science and Business Media {LLC}}
}

@ARTICLE{scipy,
  author  = {Virtanen, Pauli and Gommers, Ralf and Oliphant, Travis E. and
            Haberland, Matt and Reddy, Tyler and Cour\-na\-peau, David and
            Burovski, Evgeni and Peterson, Pearu and Weckesser, Warren and
            Bright, Jonathan and {van der Walt}, St{\'e}fan J. and
            Brett, Matthew and Wilson, Joshua and Millman, K. Jarrod and
            Mayorov, Nikolay and Nelson, Andrew R. J. and Jones, Eric and
            Kern, Robert and Larson, Eric and Carey, C J and
            Polat, {\.I}lhan and Feng, Yu and Moore, Eric W. and
            {VanderPlas}, Jake and Laxalde, Denis and Perktold, Josef and
            Cimrman, Robert and Henriksen, Ian and Quintero, E. A. and
            Harris, Charles R. and Archibald, Anne M. and
            Ribeiro, Ant{\^o}nio H. and Pedregosa, Fabian and
            {van Mulbregt}, Paul and {SciPy 1.0 Contributors}},
  title   = {{{SciPy} 1.0: Fundamental Algorithms for Scientific
            Computing in Python}},
  journal = {Nature Methods},
  year    = {2020},
  volume  = {17},
  pages   = {261--272},
  doi     = {10.1038/s41592-019-0686-2}
}

@online{openblas,
    author = "Zhang, Xianyi",
    title  = "{OpenBLAS}: An optimized BLAS library",
    note = {\url{http://www.openblas.net}. Accessed: 26.04.2024}
}

@online{splu,
  title        = {scipy.sparse.linalg.splu},
  author       = "{The SciPy community}",
  note          = {\url{docs.scipy.org/doc/scipy/reference/generated/scipy.sparse.linalg.splu.html}. Accessed: 26.04.2024}
}

@article{superlu,                                                                  
    AUTHOR = {Xiaoye S. Li},                                                    
    TITLE = {An Overview of {SuperLU}: Algorithms, Implementation,              
             and User Interface},                                               
	journal = {ACM Trans.\ Math.\ Softw.},
    volume = {31},                                                              
    number = {3},                                                               
    month = {9},                                                        
    year = {2005},                                                              
    pages = {302-325},                                                          
}

@incollection{Haasdonk:RB,
	author = {Bernard Haasdonk},
	title = {{Reduced Basis Methods for Parametrized PDEs --- A Tutorial}},
	booktitle = {Model Reduction and Approximation},
	editor = {Benner, P. and Cohen, A. and Ohlberger, M. and Willcox, K.},
	chapter = {2},
	pages = {65-136},
	publisher = {SIAM},
	address = {Philadelphia},
	year = {2017},
}

@Inbook{Urban:RB,
author="Urban, Karsten",
editor="Falcone, Maurizio
and Rozza, Gianluigi",
title="The Reduced Basis Method in Space and Time: Challenges, Limits and Perspectives",
bookTitle="Model Order Reduction and Applications: Cetraro, Italy 2021 ",
year="2023",
publisher="Springer",
address="Cham",
pages="1--72",
}

@article{haasdonk2013convergence,
  title={Convergence rates of the pod--greedy method},
  author={Haasdonk, Bernard},
  journal={ESAIM: Math. Mod. Numer. Anal.},
  volume={47},
  number={3},
  pages={859--873},
  year={2013},
  publisher={EDP Sciences}
}

@article{haasdonk2008reduced,
  title={Reduced basis method for finite volume approximationsof parametrizedlinear evolution equations},
  author={Haasdonk, Bernard and Ohlberger, Mario},
  journal={ESAIM: Math. Mod. Numer. Anal.},
  volume={42},
  number={2},
  pages={277--302},
  year={2008},
  publisher={EDP Sciences}
}
\end{document}